%% file: FuncErrorEstInverse.tex
\pdfoutput=1
\documentclass{shinyart}
\usepackage[utf8]{inputenc}
\usepackage{shinybib}
\usepackage{autonum}

\renewcommand{\phi}{\varphi}

\newcommand{\R}{\mathbb{R}}
\newcommand{\calE}{\mathcal{E}}
\newcommand{\calG}{\mathcal{G}}
\newcommand{\calI}{{\mathcal{I}}}
\newcommand{\calJ}{\mathcal{J}}

\newcommand{\calR}{\mathcal{R}}
\newcommand{\calT}{{\mathcal{T}}}
\newcommand{\calU}{\mathcal{U}}
\newcommand{\calW}{\mathcal{W}}
\newcommand{\calX}{\mathcal{X}}
\newcommand{\calY}{\mathcal{Y}}

\newcommand{\dsymmnx}{\mathrm{D^{sym}_{\|\cdot\|_\calX}}}
\newcommand{\Apa}{A^*}
\newcommand{\Bpa}{B^*}

\newcommand{\sign}{\operatorname{sign}}
\newcommand{\res}{\rho}
\newcommand{\Cint}{c_I}
\newcommand{\Cstb}{c_S}

\usepackage{functan}
\Macro{H10}{{H}^{1}_0(\Omega)}
\Macro{H1}{{H}^{1}(\Omega)}
\Macro{Hm1}{{H}^{-1}(\Omega)}
\Macro{H1s}{({H}^1(\Omega))^*}
\Macro{H2}{{H}^{2}(\Omega)}
\Macro{H12K}{{H}^{1/2}(K)}
\Macro{Wm1qp}{{W}^{-1,q'}(\Omega)}
\Macro{Wm1q}{{W}^{-1,q}(\Omega)}
\Macro{W01qp}{{W}_0^{1,q'}(\Omega)}
\Macro{W1qqpK}{{W}^{1/q,q'}(K)}
\Macro{W01q}{{W}_0^{1,q}(\Omega)}
\Macro{W2pK}{{W}^{2,p}(K)}
\Macro{L2}{{L}^{2}(\Omega)}
\Macro{Linfty}{{L}^{\infty}(\Omega)}
\Macro{Lqp}{{L}^{q'}(\Omega)}
\Macro{L1c}{{L}^{1}(\omega_c)}
\Macro{Linftyc}{{L}^{\infty}(\omega_c)}
\Macro{LinftycL1c}{\m{Linftyc},\m{L1c}}
\Macro{L2c}{{L}^{2}(\omega_c)}
\Macro{L2o}{{L}^{2}(\omega_o)}
\Macro{L2K}{{L}^{2}(K)}
\Macro{L1K}{{L}^{1}(K)}
\Macro{L1pK}{{L}^{1}(\partial K\setminus\partial\Omega)}
\Macro{LinftyK}{{L}^{\infty}(K)}
\Macro{LinftypK}{{L}^{\infty}(\partial K\setminus\partial\Omega)}
\Macro{L1F}{{L}^{1}(F\setminus\partial\Omega)}
\Macro{LinftyFF}{{L}^{\infty}(\partial K)}
\Macro{LinftyF}{{L}^{\infty}(F\setminus\partial\Omega)}
\Macro{Lq}{{L}^{q}(\Omega)}
\Macro{LqK}{{L}^{q}(K)}
\Macro{L2pK}{{L}^{2}(\partial K\setminus\partial\Omega)}
\Macro{LqpK}{{L}^{q}(\partial K)}
\Macro{LqppK}{{L}^{q'}(\partial K)}
\Macro{L2F}{{L}^{2}(F\setminus\partial\Omega)}
\Macro{L2G}{{L}^{2}(\Gamma)}
\Macro{L2S}{{L}^{2}(\Sigma)}
\Macro{Linf}{{L}^{\infty}(\Omega)}
\Macro{C}{{C}_b(\Omega)}
\Macro{Cc}{{C}_b(\omega_c)}
\Macro{Co}{{C}_b(\omega_o)}
\Macro{c}{{C}(\Omega)}
\Macro{M}{\mathcal{M}(\Omega)}
\Macro{Mc}{\mathcal{M}(\omega_c)}
\Macro{Mo}{\mathcal{M}(\omega_o)}
\Macro{McCc}{\m{Mc},\m{Cc}}
\Macro{G}{\calG}
\Macro{U}{\calU}
\Macro{X}{\calX}
\Macro{W}{\calW}
\Macro{Wtils}{\tilde{\calW}^*}
\Macro{Wtil}{\tilde{\calW}}
\Macro{Ws}{\calW^*}
\Macro{Y}{\calY}
\Macro{Ytils}{\tilde{\calY}^*}
\Macro{Ytil}{\tilde{\calY}}

\usepackage{tikz}
\usepackage{pgfplots}
\usepgfplotslibrary{colorbrewer}
\pgfplotsset{compat=newest}
\pgfplotsset{plot coordinates/math parser=false}

\hypersetup{
    pdftitle={Functional error estimators for the adaptive discretization of inverse problems},
    pdfauthor={Christian Clason, Barbara Kaltenbacher, Daniel Wachsmuth},
}
\title{Functional error estimators for the adaptive discretization of inverse problems}
\author{Christian Clason%
    \thanks{Faculty of Mathematics, Universität Duisburg-Essen, 45117 Essen, Germany (\email{christian.clason@uni-due.de})}
    \and Barbara Kaltenbacher%
    \thanks{Institute of Mathematics,
        Alpen-Adria Universität Klagenfurt,
        Universit\"atsstrasse 65--67,
        A-9020 Klagenfurt, Austria
    (\email{barbara.kaltenbacher@aau.at})}
    \and Daniel Wachsmuth%
    \thanks{Institute of Mathematics,
        Universit\"at W\"urzburg,
        Emil-Fischer-Stra{\ss}e 30,
        97074 W\"urzburg, Germany
    (\email{daniel.wachsmuth@mathematik.uni-wuerzburg.de})}
}

\begin{document}
\maketitle
\begin{abstract}
    So-called functional error estimators provide a valuable tool for reliably estimating the discretization error for a sum of two convex functions. We apply this concept to Tikhonov regularization for the solution of inverse problems for partial differential equations, not only for quadratic Hilbert space regularization terms but also for nonsmooth Banach space penalties. Examples include the measure-space norm (i.e., sparsity regularization) or the indicator function of an $L^\infty$ ball (i.e., Ivanov regularization). The error estimators can be written in terms of residuals in the optimality system that can then be estimated by conventional techniques, thus leading to explicit estimators. This is illustrated by means of an elliptic inverse source problem with the above-mentioned penalties, and numerical results are provided for the case of sparsity regularization.
\end{abstract}

\section{Introduction}

Variational regularization often leads to minimizing a sum of two convex functionals and discretization is usually performed by restricting minimization to a finite dimensional subspace.
For inverse problems in the context of large scale PDE models, adaptive  refinement of the computational mesh is crucial for an efficient numerical solution.
Recent contributions to the topic of adaptive discretization of inverse problems can be found in, e.g.,
\cite{HaberHeldmannAscher07} on adaptive finite volume discretizations for
Tikhonov--TV regularization,
\cite{KindermannNeubauer03,Neubauer07} on moving mesh regularization and adaptive grid regularization,
\cite{BenAmeurChaventJaffre02,BenAmeurKaltenbacher02,ChaventBissel98} on
refinement and coarsening indicators,
and
\cite{BangerthJoshi08,
    BeilinaClason06, BeilinaJohnson05,BeilinaKlibanov10,
    GriesbaumKaltenbacherVexler08,KaltenbacherKirchnerVexler11,
KaltenbacherKirchnerVeljovic13,KaltenbacherKirchnerVexler13} on goal oriented error estimators.

A key step for adaptive discretization is reliable estimation of the discretization error using quantities available in the numerical computations, i.e., in an \emph{a posteriori} fashion.
The functional error estimators described in \cite{Repin00} allow for an exact estimate of the discretization error and appear to be particularly promising for Tikhonov regularized inverse problems since they have originally been developed in the context of minimization of a sum of two convex functionals.
Yet so far they have not been considered for inverse problems and only very recently for control problems in, e.g., \cite{GaevskajaHoppeRepin07,LangerRepinWolfmayr14,Wolfmayr14}.
Regarding nonsmooth problems, functional error estimates have been used to derive a posteriori error estimators for the finite-element discretization of total variation denoising in \cite{Bartels:2015}.

In this work, we are concerned with linear inverse problems for PDEs consisting of the forward model
\begin{align}
    Ay&=Bu\label{model}
    \intertext{together with the measurement equation}
    Cy&=g\label{measurement}
\end{align}
where $u$ is the unknown parameter (e.g., source term, boundary data, or coefficient), $y$ is the corresponding state solving \eqref{model}, $g$ is the observable data, $A:\calY\to\calW^*$, $B:\calU\to\calW^*$, and $C:\calY\to\calG$ are linear operators, and
$\calG$, $\calU$, $\calW$, and $\calY$ are Banach spaces.

As a simple motivating example, consider the inverse problem of electroencephalography \cite{ElBadiaHaDuongIP00}, which consists in recovering the current density distribution within the brain from potential measurements on the scalp.
This can be formulated (assuming constant conductivity for simplicity) as an inverse problem for the PDE
\begin{equation}\label{eq:ex2}
    \left\{\begin{aligned}
            -\Delta y &= \chi_{\omega_c} u&&\text{ in }\Omega,\\
            \partial_\nu y&= f &&\text{ on } \partial\Omega,
    \end{aligned} \right.
\end{equation}
where $u$ is the desired current density, $\omega_c\subset\Omega$ denotes the region of interest inside the skull $\Omega$, and $f$ is the given current flux on the scalp $\partial\Omega$. The measured data is $g=y|_\Gamma$, where $\Gamma\subset\partial\Omega$ denotes the location of the electrodes on the scalp.
Here, $A$ is the negative Laplace operator, $B$ is the extension operator from $\omega_c$ to $\Omega$, and $C$ is the Dirichlet trace operator on $\Gamma$.

In practice, only a noisy measurement $g^\delta$ will typically be available,  where the noise level $\delta$ defined by
\begin{equation}
    \|g-g^\delta\|_\calG\leq\delta
\end{equation}
we here assume to be known. Since the solution of such an inverse problems is typically unstable, regularization needs to be employed; see, e.g., \cite{EnHaNe96,Scherzeretal2009} and the references therein. We will here consider the classical Tikhonov--Philips method in Banach spaces with Morozov's discrepancy principle as a regularization parameter choice strategy.

Using the parameter-to-state mapping
\begin{equation}
    S:=A^{-1}B\in L(\calU,\calY)
\end{equation}
and the reduced forward operator
\begin{equation}
    K:=CS\in L(\calU,\calG),
\end{equation}
we can write (\refeq{model}--\refeq{measurement}) as a single operator equation
\begin{equation}\label{Kug}
    Ku=g.
\end{equation}
For this reduced formulation, Tikhonov's method is given by
\begin{equation}\label{eq:minred}
    \begin{aligned}
        &\min_{u\in \calU} J_\alpha(u,Ku) \quad \text{ where }\quad
        J_\alpha(u,g)=G(g)+\calR_\alpha(u),
    \end{aligned}
\end{equation}
where $\calR_\alpha$ is an appropriate regularizing functional and $G$ a discrepancy term, which in this work will be assumed to have the form
\begin{equation}\label{eq:G}
    G(g)=\frac12 \norm[]{G}{g-g^\delta}^2.
\end{equation}
The discrepancy principle (or rather its relaxed version) amounts to choosing $\alpha=\alpha(\delta)$ such that
\begin{equation}\label{eq:discrprinc_red}
    \underline{\tau}\delta \leq \norm[]{G}{Ku_{\alpha}^\delta-g^\delta} \leq
    \overline{\tau}\delta
\end{equation}
holds, where $u_{\alpha}^\delta$ is a minimizer of \eqref{eq:minred} and $\overline{\tau}\geq\underline{\tau}\geq1$ are fixed constants independent of $\delta$.
Convergence of this method has been extensively investigated in the literature; see, e.g., \cite{EnHaNe96} and the references therein for an analysis in Hilbert spaces and \cite{BurgerOsher04,Scherzeretal2009,SeidmanVogel89} for a more general setting similar to the one considered here. For actual numerical computations, the infinite-dimensional problem has to be discretized: Finite-dimensional spaces $\calU_h\subset\calU$, $\calY_h\subset\calY$, and $\calW_h\subset\calW$
are chosen, and the solution of $Ay=Bu$ is replaced by finding $y_h\in \calY_h$ such that
\begin{equation}\label{eq:weakdiscreteformulation}
    \langle Ay_h - Bu, \ w_h\rangle = 0 \qquad\forall w_h \in \calW_h.
\end{equation}
To carry the convergence results over from the infinite-dimensional to the discretized problem, the error due to discretization has to be assessed.
In particular, it is important to carefully balance discretization and regularization.
As it turns out, only errors in the functionals $G$ and $J_\alpha$ need to be controlled in order to obtain a convergent adaptive method.
This makes the theory of functional error estimators in \cite{Repin00} applicable.
As we will show, these estimators are applicable for different choices of regularization functionals.
These include the usual squared Hilbert-space norm, i.e. $\calR_\alpha=\frac{\alpha}{2}\norm[]{U}{\cdot}^2$,
but also nonsmooth penalties of the form $\calR_\alpha=\delta_{B_{1/\alpha}^{L^\infty(\omega_c)}}$ or $\calR_\alpha=\alpha\norm[]{Mc}{\cdot}$,
where
$\mathcal{M}(\Omega)$
is the space of Radon measures. The latter penalty is useful for incorporating sparsity regularization, while the former penalty corresponds to Ivanov regularization (also called method of quasi-solutions, see \cite{Ivanov62, Ivanov63, IvanovVasinTanana02, LorenzWorliczek13, SeidmanVogel89},
as well as \cite{NeubauerRamlau14} in the context of Hilbert scales), where the regularization does not take the usual additive form with $\alpha$ as a multiplier.
In all these cases, the functional error estimators can be computed in terms of residuals in the optimality system.

This work is organized as follows. After fixing some common notation, we present in \cref{sec:setting} the basic results on convergence of adaptively discretized regularization methods and the functional error estimates our analysis relies on. These estimators are then applied to the classical Hilbert space regularization in \cref{sec:errest_red_Hil}, first in the general setting and then specifically for a model inverse source problem for the Poisson equation. Similarly, \cref{sec:errest_red_Ban1} and \cref{sec:errest_red_Ban2} treat the case of Banach space norm constraints and norm regularization, respectively, again both in the general setting and for model problems (Ivanov regularization resp.~sparsity). For the latter, numerical experiments given in \cref{sec:NumTests} demonstrate the efficiency of our approach.

\section{Notation and preliminary results}\label{sec:setting}

For some Banach space $X$ with dual $X^*$, we use the notation $\langle x^*,x\rangle_{X^*,X}$ for the canonical duality pairing.
In case of a Hilbert space $X$, $(x_1,x_2)_X$ denotes the inner product.
Moreover, $\delta_S$ denotes the indicator function of some set $S$ and $B_r^X$ the closed ball of radius $r$ around zero in the normed space $X$.

\subsection{Functional-analytic setting}\label{sec:spaces}

In the following, we assume that $\calU$, $\calW$, $\calX$, $\calY$ are Banach spaces
with $\calW$ and $\calY$ being reflexive,
and that $\calG$ is a Hilbert space.
Furthermore, we suppose that either $\calX=\calU^*$  or $\calU=\calX^*$ holds, which allows us to use a consistent notation in the rest of the paper and to avoid cumbersome case distinctions.
For a convex functional $F:\calU\to\bar{\R}$, we will denote by
\begin{equation}
    F^*:\calX\to\bar{\R},\qquad F^*(x) = \sup_{u\in \calU} \langle u,x\rangle_{\calU,\calX} - F(u)
\end{equation}
its Fenchel conjugate.
If $\calX = \calU^*$, this coincides with the usual definitions in the sense of convex analysis.
For $\calU = \calX^*$, it is common to define as here the Fenchel conjugate on $\calX$ instead of $\calX^{**}$ in the special case of $F=G^*$ (i.e., the biconjugate of $G$); the redefinition in the general case is less common but still consistent and coincides with the ``predual'' approach as in, e.g., \cite{Clason:2009}. This will allow working with spaces of continuous functions instead of the dual of measure spaces later on. In particular, the Fenchel conjugate of  $F(u) = \alpha\norm[]{U}{u}$ is always given by
\begin{equation}
    F^*(x) = \delta_{B^X_\alpha}(x) =
    \begin{cases}
        0 & \text{if }\norm[]{X}{x}\leq \alpha,\\
        \infty & \text{if }\norm[]{X}{x} >\alpha.
    \end{cases}
\end{equation}
In the case that $\calU$ is a Hilbert space, we set $\calX=\calU$, in which case the duality pairing coincides with the standard inner product. In particular, for $F(u)=\frac12 \norm[]{U}{u-z}^2$ we have
\begin{equation}
    F^*(u)=\frac12 \left(\norm[]{U}{ u-z}^2-\norm[]{U}{z}^2\right).
\end{equation}

We further denote by
\begin{equation}
    \partial F(u) := \left\{x\in \calU^*: \langle \tilde u - u,x\rangle_{\calU,\calX} \leq F(\tilde u)- F(u) \quad\text{for all } \tilde u\in \calU\right\}
\end{equation}
the convex subdifferential of $F:\calU\to\bar{\R}$. Note that we always have the inclusion $\calX\subset\calU^*$, either by equality or by using the canonical injection from $\calX$ to $\calX^{**}$.
In the latter case, existence of the duality mapping $\calJ^\calU:\calU\to \calX$, defined by
\begin{equation}
    \norm[]{X}{\calJ^\calU(u)}=1 \quad\text{ and }\quad
    \langle u,\calJ^\calU(u)\rangle_{\calU,\calX}  =\norm[]{U}{u}\qquad\text{for all }u\in\calU,
\end{equation}
i.e., $\calJ^\calU(u) \in\partial(\norm[]{U}{\cdot})(u)$, becomes an additional assumption.

We further need the linear operators $A\in L(\calY,\calW^*)$, $B\in L(\calU,\calW^*)$, and $C\in L(\calY,\calG)$, and assume that $A$ is continuously invertible.
We will also make use of the adjoints
\begin{align}
    \Apa&\in L(\calW,\calY^*)&&\text{ with }&
    \langle A y, w\rangle_{\calW^*,\calW}&=\langle y, \Apa w\rangle_{\calY,\calY^*}
                                         &&\text{ for all  } y\in \calY, \ w\in \calW,\\
    \Bpa&\in L(\calW,\calX)&&\text{ with }&
    \langle B u, w\rangle_{\calW^*,\calW}&=\langle u, \Bpa w\rangle_{\calU,\calX}
                                         &&\text{ for all  } u\in \calU, \ w\in \calW,\\
    C^*&\in L(\calG,\calY^*)&&\text{ with }&
    ( C y, g)_\calG&=\langle y, C^*g^*\rangle_{\calY,\calY^*}
                   &&\text{ for all  } y\in \calY, \ g\in \calG,\\
    K^*&\in L(\calG,\calX)&&\text{ with }&
    ( K u, g)_\calG&=\langle u, K^*g^*\rangle_{\calU,\calX}
                   &&\text{ for all  } u\in \calU, \ g\in \calG,\\
    S^*&\in L(\calY^*,\calX)&&\text{ with }&
    \langle S u, y^*\rangle_{\calY,\calY^*}&=\langle u, S^*y^*\rangle_{\calU,\calX}
                                           &&\text{ for all  } u\in \calU, \ y^*\in \calY^*.
\end{align}
Let us emphasize that the existence of $B^*$ with the mentioned properties is
an actual assumption in the case $\calU = \calX^*$, which is equivalent
to the assumption that $B$ is the adjoint operator of an operator ${}^*B$.
(With a slight abuse of notation in the first two cases, since these are actually the compositions of the standard adjoints with the canonical embeddings $W\to W^{**}$).
In addition, $\{\calR_\alpha\}_{\alpha>0}$, $\calR_\alpha:\calU\to\bar\R$, is a family of proper, convex, lower semicontinuous functionals.

Finally, let $\calU_h$, $\calY_h$, $\calW_h$ be finite dimensional subspaces of $\calU$, $\calY$, $\calW$, respectively.
In the case that $\calU$ is a Hilbert space, we will denote by $P_{\calU_h}$ the orthogonal projection onto $\calU_h$. Furthermore, $R_{\calW_h}:\calW^*\to\calW_h^*$ and $R_{\calY_h}:\calY^*\to\calY_h^*$ denote the Ritz projectors defined by
\begin{equation}\label{RYhRWh}
    \langle R_{\calW_h}w^*,w_h\rangle_{\calW_h^*,\calW_h}=\langle w^*,w_h\rangle_{\calW^*,\calW}, \qquad
    \langle R_{\calY_h}y^*,y_h\rangle_{\calY_h^*,\calY_h}=\langle y^*,y_h\rangle_{\calY^*,\calY}.
\end{equation}

\subsection{Convergence of adaptively discretized Tikhonov regularization} \label{subsec:Tikhonov}

We consider the Tikhonov regularization \eqref{eq:minred} equivalently written as a PDE-constrained minimization problem
\begin{equation}\label{eq:Tikh}
    \begin{aligned}
        &\min_{u\in \calU, \ y\in \calY} J_\alpha(u,y):= \frac12 \norm[]{G}{Cy-g^\delta}^2 +\calR_\alpha(u)
        \quad\text{ s.t. }\quad
        Ay=Bu \text{ in }\calW^*.
    \end{aligned}
\end{equation}
The discrete counterpart of \eqref{eq:Tikh} reads
\begin{equation}\label{eq:Tikh_h}
    \begin{aligned}
        &\min_{u\in \calU_h, \ y\in \calY_h} J_\alpha (u,y)
        \quad
        \text{ s.t. }\quad
        R_{\calW_h} (Ay-Bu)=0.
    \end{aligned}
\end{equation}
Let $(u_{\alpha}^\delta,y_{\alpha}^\delta)$ be the exact Tikhonov minimizer, i.e., a solution of \eqref{eq:Tikh},
and let $(u_h,y_h) \in \calU_h\times\calY_h$ be some approximation, e.g., a solution of the discrete problem \eqref{eq:Tikh_h}.
In this abstract setting we just presume existence of minimizers of \eqref{eq:Tikh} and \eqref{eq:Tikh_h} and will verify this assumption for the applications in \cref{subsec:example1}, \cref{subsec:example2}, and \cref{subsec:example3}.
The question is now how the convergence of the discrete approximation $u_h$ to solutions of the equation $Ku=g$ can be guaranteed for $(h,\alpha,\delta)\searrow0$.

The following theorem shows (similarly as in \cite{KaltenbacherKirchnerVexler11,NeuSch90}) that it is enough to adapt the discretization
and the choice of the regularization parameter $\alpha(\delta,h)$
in such a way that the
difference in the functional values satisfies
\begin{equation}\label{etaJ}
    J_\alpha(u_h,y_h)-J_\alpha(u^\delta_\alpha,y_{\alpha}^\delta)\leq\eta_J,
\end{equation}
and the difference in the discrepancy values satisfies
\begin{equation}\label{etaD}
    \norm[]{G}{K_h{u_h}-g^\delta}^2-\norm[]{G}{K{u^\delta_\alpha}-g^\delta}^2
    =\norm[]{G}{Cy_h-g^\delta}^2-\norm[]{G}{C{y^\delta_\alpha}-g^\delta}^2\leq \eta_D,
\end{equation}
where $\eta_J$ and $\eta_D$ can be controlled to be small enough relative to $\delta$.
\begin{proposition}\label{prop:conv}
    Let $(u_\alpha^\delta,y_\alpha^\delta)$ be a minimizer of \eqref{eq:Tikh} and  $(u_{\alpha,h}^\delta,y_{\alpha,h}^\delta)$ be a minimizer of \eqref{eq:Tikh_h}.
    Let $\alpha(\delta)$ be chosen such that for some constants $c_1,c_2,\overline{\tau},\underline{\tau}>0$ independent of $\delta$ with $\overline{\tau}>\underline{\tau}\geq\max\{\sqrt{1+2c_2},1+c_1\}$, the estimates
    \begin{align}
        \underline{\tau}\delta \leq \norm[]{G}{Cy_{\alpha(\delta),h}^\delta-g^\delta} &\leq
        \overline{\tau}\delta,\label{eq:discrprinc}\\
        \left|\norm[]{G}{Cy_{\alpha(\delta),h}^\delta-g^\delta}-\norm[]{G}{Cy_{\alpha(\delta)}^\delta-g^\delta}\right|&\leq c_1\delta,\label{eq:accuracy_residual}\\
        \intertext{and}
        J_{\alpha(\delta)}(u_{\alpha(\delta),h}^\delta,y_{\alpha(\delta),h}^\delta)
        -J_{\alpha(\delta)}(u_{\alpha(\delta)}^\delta,y_{\alpha(\delta)}^\delta)&\leq c_2\delta^2
        \label{eq:accuracy_cost}
    \end{align}
    hold.
    Then for any solution $u^\dagger$ to $Ku=g^\dagger$, we have
    \begin{equation}\label{eq:Rbdd}
        \calR_{\alpha(\delta)}(u_{\alpha(\delta)}^\delta) \leq \calR_{\alpha(\delta)}(u^\dagger)
        \quad\text{ and }\quad
        \calR_{\alpha(\delta)}(u_{\alpha(\delta),h}^\delta) \leq \calR_{\alpha(\delta)}(u^\dagger)
        \quad  \text{ for all } \delta>0.
    \end{equation}
    Moreover, we have
    \begin{equation} \label{eq:resOdelta}
        \norm[]{G}{Cy_{\alpha(\delta),h}^\delta-g^\delta} \leq \overline{\tau}\delta\to 0
        \quad\text{ and }\quad
        \norm[]{G}{Cy_{\alpha(\delta)}^\delta-g^\delta} \leq (\overline{\tau}+c_1)\delta\to 0.
    \end{equation}
\end{proposition}
\begin{proof}
    Set $\alpha_*:=\alpha(\delta)$.
    By the assumptions (\refeq{eq:discrprinc}--\refeq{eq:accuracy_cost}) and minimality of
    $(u_{\alpha_*}^\delta, y_{\alpha_*}^\delta)$, we have for any solution $u^\dagger$ to $Ku=g^\dagger$
    \begin{equation}
        \begin{aligned}
            {\frac12\underline{\tau}^2\delta^2 + \calR_{\alpha_*}(u_{\alpha_*,h}^\delta)-c_2\delta^2}
            &\leq \frac12\norm[]{G}{Cy_{\alpha_*,h}^\delta-g^\delta}^2+\calR_{\alpha_*}(u_{\alpha_*,h}^\delta)
            -c_2\delta^2\\
            &\leq \frac12\norm[]{G}{Cy_{\alpha_*}^\delta-g^\delta}^2+\calR_{\alpha_*}(u_{\alpha_*}^\delta)
            \leq \frac12\norm[]{G}{Ku^\dagger-g^\delta}^2+\calR_{\alpha_*}(u^\dagger)
            \\
            &\leq \frac12\delta^2+\calR_{\alpha_*}(u^\dagger)
            \leq \frac12\frac{1}{(\underline{\tau}-c_1)^2} \norm[]{G}{Cy_{\alpha_*}^\delta-g^\delta}^2
            +\calR_{\alpha_*}(u^\dagger)
        \end{aligned}
    \end{equation}
    (where we have used
    $\norm[]{G}{Cy_{\alpha_*}^\delta-g^\delta}\geq (\underline{\tau}-c_1)\delta$ in the last estimate),
    which by comparison of the third and the sixth as well as of the first and the fifth expression in this chain of inequalities together with $\underline{\tau}\geq\max\{\sqrt{1+2c_2},1+c_1\}$ yields
    \eqref{eq:Rbdd}. The convergence \eqref{eq:resOdelta} follows directly from \eqref{eq:discrprinc} and \eqref{eq:accuracy_residual}.
\end{proof}
Note that no absolute value is required in the estimate \eqref{eq:accuracy_cost}.
From \eqref{eq:Rbdd} and \eqref{eq:resOdelta}, convergence and convergence rates for both the continuous and discrete sequence as $\delta \to 0$ follow under the usual assumptions on $\calR$, see, e.g., \cite{EnHaNe96,Scherzeretal2009,SKHK12}.

\begin{remark}
    Here we have taken into account the fact that in practical computations, the discrepancy principle \eqref{eq:discrprinc} can only be checked for the discrete residual
    $\norm[]{G}{Cy_{\alpha_*(\delta),h}^\delta-g^\delta}=\norm[]{G}{K_hu_{\alpha_*(\delta),h}^\delta-g^\delta}$, not the exact residual $\norm[]{G}{Ku_{\alpha_*(\delta),h}^\delta-g^\delta}$ for which \eqref{etaD} can be employed. To bridge the gap between these two quantities, we will use the triangle inequality and an additional estimate of
    $\norm[]{G}{K_hu_{\alpha_*(\delta),h}^\delta-Ku_{\alpha_*(\delta),h}^\delta}$.
\end{remark}

The accuracy requirements that will have to be met by an adaptive discretization are stated in assumptions
\eqref{eq:accuracy_residual} and \eqref{eq:accuracy_cost}. Note that for this purpose, the accuracy of $u$ need not be controlled directly,
but only via the residual norm and cost function values.
In the next section, we will derive
corresponding estimates based on the functional error estimates from \cite{Repin00}.

\subsection{Functional a posteriori estimators}\label{subsec:Repin}

Our approach is based on the following functional error estimate,
which is inspired by \cite{Repin00}.
We employ the strong convexity of the discrepancy term \eqref{eq:G}
to obtain a slightly improved estimate.

\begin{proposition}\label{th:Repin}
    Let $(u_\alpha^\delta,y_\alpha^\delta)$ be a minimizer of \eqref{eq:Tikh}.
    Assume that there is a family of functions $\{\phi_\alpha\}_{\alpha>0}$, $\phi_\alpha:\calU \times \calU \to\R_0^+$,
    satisfying
    \begin{equation}\label{eq:strongconvexity}
        \lambda(1-\lambda)\phi_\alpha(u_1,u_2)\leq \lambda \calR_\alpha(u_1)+(1-\lambda)\calR_\alpha(u_2)- \calR_\alpha\left(\lambda u_1+(1-\lambda)u_2\right)
    \end{equation}
    for all $u_1,u_2\in \calU$, $\alpha>0$, and $\lambda\in(0,1)$.
    Let $v\in \calU$ and $g^*\in \calG$ be arbitrary.
    Then, any $v\in \calU$ and $g^*\in \calG$ satisfy
    \begin{equation}
        \begin{aligned}[t]
            \frac12\| K(u_\alpha^\delta - v)\|_\calG^2  + \phi_\alpha(u_\alpha^\delta, v)
            &\le J_\alpha(v,Kv) - J_\alpha(u_\alpha^\delta,Ku_\alpha^\delta)\\
            &\le \calR_\alpha(v)+\calR_\alpha^*(K^*g^*)+G(Kv)+ G^*(-g^*).
        \end{aligned}
        \label{eq:estRepin}
    \end{equation}
\end{proposition}
\begin{proof}
    Due to the assumptions and the strong convexity of $G$, we have
    for $v\in \calU$ and $\lambda\in(0,1)$
    \begin{equation}
        \begin{aligned}
            \lambda(1-\lambda)\left( \frac12 \|K(u_\alpha^\delta - v)\|_\calG^2
            + \phi_\alpha(u_\alpha^\delta, v) \right)
            &\leq \lambda J_\alpha(u_\alpha^\delta,Ku_\alpha^\delta)+(1-\lambda)J_\alpha(v,Kv)\\
            \MoveEqLeft[-1]- J_\alpha\left(\lambda u_\alpha^\delta+(1-\lambda)v,K(\lambda u_\alpha^\delta+(1-\lambda)v)\right)\\
            &\leq (1-\lambda)(J_\alpha(v,Kv)-J_\alpha(u_\alpha^\delta,Ku_\alpha^\delta)),
        \end{aligned}
    \end{equation}
    where we have used optimality of $u_\alpha^\delta$ in the last step.
    Dividing by $1-\lambda$ and letting $\lambda\nearrow 1$, we obtain the first inequality. The second inequality is a consequence of weak duality.
\end{proof}
Condition~\eqref{eq:strongconvexity} is satisfied, e.g., with
$\varphi_\alpha(u_1,u_2)=\frac{\alpha}{2}\norm[]{U}{u_1-u_2}^2$
in the case of a quadratic Hilbert space penalty; see \cref{sec:errest_red_Hil}.
But we will see that \eqref{eq:estRepin} still provides valuable information on the error if \eqref{eq:strongconvexity} is only satisfied with $\varphi_\alpha(u_1,u_2)=0$, as in the case of Banach space norm constraints and penalties; see \cref{sec:errest_red_Ban1} and \cref{sec:errest_red_Ban2}, respectively.

Here it is important to note that the right-hand side of estimate \eqref{eq:estRepin}
does not contain the unknown solution $u_\alpha^\delta$. We will use this estimate
with $v:=u_{\alpha,h}^\delta$, which is available in the numerical computations.
We also point out that the right-hand side corresponds to the duality gap between problem \eqref{eq:minred} and its dual problem in the sense of convex analysis; see, e.g., \cite{EkelandTemam}. Hence if $v$ and $g^*$ satisfy primal-dual extremality relations for \eqref{eq:minred}, then the right-hand side of \eqref{eq:estRepin} vanishes.

The sub- and superscripts $\alpha$, $\delta$ will be omitted in the following.
Instead, we will write $(\bar{u}, \bar{y})$, $(\bar{u}_h, \bar{y}_h)$ for the continuous and discrete minimizers $(u_\alpha^\delta,y_\alpha^\delta)$, $(u_{\alpha,h}^\delta,y_{\alpha,h}^\delta)$, respectively.

\subsection{Model problem}

To illustrate the derived estimates, we will apply them to the identification of the source term $u$ in
\begin{equation}\label{eq:ex1}
    \left\{\begin{aligned}
            -\Delta y &= \chi_{\omega_c}u &&\text{ in }\Omega,\\
            y&= 0 &&\text{ on } \partial\Omega,
    \end{aligned} \right.
\end{equation}
on a domain $\Omega\subseteq\R^n$, $n\in\{1,2,3\}$, from restricted observations $g^\delta$ of $y$ in $\omega_o$.
Hence,
\begin{equation}
    \left\{\begin{aligned}
            Ay&=-\Delta y, \quad &A^*&=A, \\
            Bu&= \chi_{\omega_c}u, \quad &\Bpa w &= w\vert_{\omega_c},\\
            Cy&=y\vert_{\omega_o}, \quad &C^* g&= \chi_{\omega_o} g,\\
    \end{aligned}\right.
\end{equation}
and $\calG=L^2(\omega_o)$.
In the sequel, we assume that $\Omega$ is polyhedral and convex. This enables us to employ $H^2$-regularity results for the elliptic equation \eqref{eq:ex1}.
In addition, we can avoid technicalities in the finite element setting on curved domains.

We define $\calY_h=\calW_h$ by continuous piecewise linear finite elements on a shape regular triangulation $\calT_h$
consisting of element domains $K$;
see, e.g., \cite{Braess07}.
The set of all faces of elements will be denoted by $\calE_h$.
The associated nodal interpolation operator will be denoted by $\calI^{\calT}{}\!$, which is continuous from ${C}_b(\overline{\Omega})$ to $\calY_h$.
We will employ the standard interpolation estimates
\begin{equation} \label{eq:estint}
    \left\{\begin{aligned}
            &\forall K\in \calT_h:&\norm[]{L2K}{v-\calI^\calT v}&\leq \Cint h_K^2 |v|_{H^2(K)} \quad &&\forall v\in H^2(\Omega)\,,\\
            &\forall F\in \calE_h:&\norm[]{L2F}{v-\calI^\calT v}&\leq \Cint h_K^{3/2} |v|_{H^2(K)} \quad &&\forall v\in H^2(\Omega)\,,
    \end{aligned}\right.
\end{equation}
where $h_K$ is the element diameter, as well as the stability estimate
\begin{equation} \label{eq:eststab}
    \norm[]{H2}{v} \leq \Cstb \norm[]{L2}{\Delta v}\qquad \forall v\in H^2(\Omega)\cap H_0^1(\Omega),
\end{equation}
cf. \cite[Thm.~II.6.4]{Braess07} and \cite[Thm.~3.3.7]{Ciarlet78}.

\section{Hilbert space regularization}\label{sec:errest_red_Hil}

In this section, we assume that $\calU$ is a Hilbert space, identify $\calX$ with $\calU$, and consider as regularization term the squared norm, i.e.,
\begin{equation}
    \begin{aligned}
        \calR_\alpha=\frac{\alpha}{2}\norm[]{U}{\cdot}^2,  \qquad \text{and hence} \qquad
        \calR_\alpha^*=\frac{1}{2\alpha}\norm[]{U}{\cdot}^2.
    \end{aligned}
\end{equation}
Since $\calJ_\alpha$ is differentiable, we obtain for \eqref{eq:Tikh} by standard Lagrangian calculus the optimality system
\begin{equation}\label{eq:optsys}
    \left\{\begin{aligned}
            C^*(C\bar{y}-g^\delta)+\Apa\bar{w}&=0,\\
            \alpha \bar{u}- \Bpa\bar{w}&=0,\\
            A\bar{y}-B\bar{u}&=0.
    \end{aligned}\right.
\end{equation}
The corresponding discrete system for \eqref{eq:Tikh_h} is\begin{equation}\label{eq:optsys_discr}
    \left\{\begin{aligned}
            R_{\calY_h} \left(C^*(C\bar{y}_h-g^\delta)+ \Apa\bar{w}_h\right)&=0,\\
            \alpha \bar{u}_h-P_{\calU_h} \Bpa\bar{w}_h&=0,\\
            R_{\calW_h} \left(A\bar{y}_h-B\bar{u}_h\right)&=0,
    \end{aligned}\right.
\end{equation}
with $R_{\calY_h}$, $R_{\calW_h}$ as in \eqref{RYhRWh},
which corresponds to a finite element discretization of the state and adjoint equation.
The solution $(\bar{u}_h,\bar{y}_h,\bar{w}_h) \in \calU_h\times \calY_h\times \calW_h$ of \eqref{eq:optsys_discr} can be considered as an approximation to the solution $(\bar{u},\bar{y},\bar{w})\in\calU\times\calY\times\calW$ of \eqref{eq:optsys}.

\subsection{Error estimates}

Setting $\phi_\alpha(u_1,u_2)=\frac{\alpha}{2} \norm[]{U}{u_1-u_2}^2$,
we obtain from \cref{th:Repin} that the solution $\bar{u}$ to \eqref{eq:minred}
satisfies
\begin{equation}\label{eq:est_red}
    \begin{aligned}[t]
        \alpha \norm[]{U}{u-\bar{u}}^2+\norm[]{G}{Ku-K\bar{u}}^2
        &\leq2\left(J_\alpha(u,Ku)-J_\alpha(\bar{u},K\bar{u})\right)\\
        &\leq\alpha \norm[]{U}{u}^2
        +\frac{1}{\alpha}\norm[]{U}{K^* g^*}^2
        +\norm[]{G}{Ku-g^\delta}^2+\norm[]{G}{g^*-g^\delta}^2-\norm[]{G}{g^\delta}^2
        \\
        &= \alpha \norm[]{U}{u}^2
        +\frac{1}{\alpha}\norm[]{U}{K^*(g-g^\delta)}^2
        +\norm[]{G}{Ku-g^\delta}^2+\norm[]{G}{g}^2-\norm[]{G}{g^\delta}^2
        \\
        &=
        \frac{1}{\alpha} \norm[]{U}{\alpha u + K^*(g-g^\delta)}^2
        +\norm[]{G}{Ku-g}^2,
    \end{aligned}
\end{equation}
for any $u\in \calU$ and $g^*:=g^\delta -g \in \calG$ for any $g\in \calG$.
We now define
\begin{equation}\label{eq:yhat}
    \hat{y}:=S\bar{u}_h=A^{-1}B\bar{u}_h.
\end{equation}
Inserting $u=\bar{u}_h$ and $g=C\bar{y}_h$ in \eqref{eq:est_red},
we arrive at
\begin{equation}\label{eq:est_red_h_aposteriori}
    \begin{aligned}[t]
        \alpha \norm[]{U}{\bar{u}_h-\bar{u}}^2+ \norm[]{G}{C\hat{y}-C\bar{y}}^2
        &\leq 2 \left(J_\alpha(\bar{u}_h,\hat{y})-J_\alpha(\bar{u},\bar{y})\right)\\
        &\leq   \frac{1}{\alpha} \norm[]{U}{\alpha \bar{u}_h
        +  S^* C^* (C\bar{y}_h-g^\delta)}^2
        +\norm[]{G}{C(A^{-1}B\bar{u}_h-\bar{y}_h)}^2
        \\
        &=
        \frac{1}{\alpha} \norm[]{U}{\alpha \bar{u}_h- \Bpa\bar{w}_h \
        +\  S^*\left(C^*(C\bar{y}_h-g^\delta)+ \Apa\bar{w}_h\right)}^2\\
        \MoveEqLeft[-1]
        +\norm[]{G}{CA^{-1}\left(A\bar{y}_h-B\bar{u}_h\right)}^2.
    \end{aligned}
\end{equation}
Here, \eqref{eq:est_red_h_aposteriori} contains the residuals of the equations in the optimality system \eqref{eq:optsys}, which are given by
\begin{equation}\label{eq:res123}
    \left\{\begin{aligned}
            \res_w&:=C^*(C\bar{y}_h-g^\delta)+\Apa\bar{w}_h=\Apa(\bar{w}_h-\hat{w}),\\
            \res_u&:=\alpha \bar{u}_h- \Bpa\bar{w}_h,\\
            \res_y&:=A\bar{y}_h-B\bar{u}_h=A(\bar{y}_h-\hat{y}),
    \end{aligned}\right.
\end{equation}
where $(\hat{y},\hat{w})\in\calY\times\calW$ and $(\bar{y}_h,\bar{w}_h)\in\calY_h\times\calW_h$ satisfy
\begin{equation}\label{eq:yhatwhat}
    \left\{\begin{aligned}
            R_{\calY_h} \left(C^*(C\bar{y}_h-g^\delta)+ \Apa\bar{w}_h\right)&=0,\\
            R_{\calW_h} \left(A\bar{y}_h-B\bar{u}_h\right)&=0,
    \end{aligned}\right.
    \qquad
    \left\{\begin{aligned}
            C^*(C\bar{y}_h-g^\delta)+ \Apa\hat{w}&=0,\\
            A\hat{y}-B\bar{u}_h&=0,
    \end{aligned}\right.
\end{equation}
for the same $\bar{u}_h\in\calU_h$ (note that the left system is coupled, as opposed to the right one).
Thus the inequality \eqref{eq:est_red_h_aposteriori} appears to be suited for a posteriori error estimation.

Although estimate \eqref{eq:est_red_h_aposteriori} only gives an estimate on $K\bar{u}_h-K\bar{u}=C\hat{y}-C\bar{y}$ and not on $K_h\bar{u}_h-K\bar{u}=C\bar{y}_h-C\bar{y}$ (which is needed for \eqref{eq:accuracy_residual}), we can use the identity $\bar{y}_h-\hat{y}=A^{-1}\res_y$,i.e.,
\begin{equation}\label{CAres3}
    C\bar{y}_h-C\hat{y}=CA^{-1}\left(A\bar{y}_h-B\bar{u}_h\right),
\end{equation}
the triangle inequality, and the fact that
\begin{equation}\label{abcd}
    \forall a,b,c,d\geq0 : \ a+b^2\leq c+d^2 \ \Rightarrow \ a+(b+d)^2\leq \gamma c+\sigma d^2
\end{equation}
holds for
\begin{equation}\label{rhosigma}
    \left(\sigma=4\text{ and }\gamma\geq2\right) \quad\text{ or }
    \quad\left(\sigma>4\text{ and }\gamma>\frac{2\sigma}{\sigma+\sqrt{\sigma^2-4\sigma}}\right)
\end{equation}
(see the \nameref{app} for a proof)
as well as
\begin{equation}\label{estJ}
    J_\alpha(\bar{u}_h,\bar{y}_h)-J_\alpha(\bar{u}_h,\hat{y})
    = (C\bar{y}_h-g^\delta,C\bar{y}_h-C\hat{y})_\calG-\frac12\norm[]{G}{C\bar{y}_h-C\hat{y}}
\end{equation}
to obtain from \eqref{eq:est_red_h_aposteriori} the following a posteriori estimate.
\begin{proposition}\label{aposteriori_hilbert}
    Let $\calU$ be a Hilbert space and $\calR_\alpha=\frac\alpha2\norm[]{U}{\cdot}^2$. Then the  minimizers $(\bar{u},\bar{y})$ of \eqref{eq:Tikh} and  $(\bar{u}_h,\bar{y}_h)$ of  \eqref{eq:Tikh_h} satisfy the estimates
    \begin{align}
        \qquad\alpha \norm[]{U}{\bar{u}_h-\bar{u}}^2+ \norm[]{G}{C\bar{y}_h-C\bar{y}}^2
        &\leq\frac{\gamma}{\alpha} \norm[]{U}{ \Bpa (\Apa)^{-1} \res_w\ + \ \res_u}^2
        +\sigma\norm[]{G}{CA^{-1}\res_y}^2,\label{eq:est_red_h_aposteriori_s}\\
        J_\alpha(\bar{u}_h,\bar{y}_h)-J_\alpha(\bar{u},\bar{y})
        &\leq\frac{1}{2\alpha} \norm[]{U}{ \Bpa (\Apa)^{-1} \res_w\ + \ \res_u}^2
        +
        (C\bar{y}_h-g^\delta,CA^{-1}\res_y)_\calG,
        \label{eq:est_red_h_aposteriori_J}
    \end{align}
    with $\sigma$ and $\gamma$ as in \eqref{rhosigma} and $\res_w$, $\res_u$, and $\res_y$ as in \eqref{eq:res123}.
\end{proposition}
Here the factors $\sigma$ and $\gamma$ may be used to minimize the right hand side of the estimate. In the following, we will fix $\sigma=4$, $\gamma=2$ for simplicity.

At a first glance, estimate \eqref{eq:est_red_h_aposteriori} requires solution of state and adjoint equation on a fine grid for applying
$S^*$
and $CA^{-1}$, but this can be avoided in some relevant examples; see, e.g., \cref{subsec:example1} below.

\subsection{Application to inverse source problem}
\label{subsec:example1}

We now apply the estimate from \cref{aposteriori_hilbert} to the model problem \eqref{eq:ex1}. In this case, we have $\calU=L^2(\omega_c)$ as well as $\calY= H_0^1(\Omega)=\calW$, and the Tikhonov problem is given by
\begin{equation}\label{eq:ex1min}
    \left\{\begin{aligned}
            &\min_{y,u} \frac12 \norm[]{L2o}{y-g^\delta}^2 +\frac{\alpha}{2}\norm[]{L2c}{u}^2\\
            &\mbox{s.t. }-\Delta y = \chi_{\omega_c} u,\quad y|_{\partial\Omega}=0.
    \end{aligned}\right.
\end{equation}
Hence, using
\begin{align}
    \res_w&=\chi_{\omega_o}(\bar{y}_h-g^\delta)-\Delta\bar{w}_h,\\
    \res_u&=\alpha \bar{u}_h-\bar{w}_h\vert_{\omega_c},\\
    \res_y&=-\Delta\bar{y}_h-\chi_{\omega_c}\bar{u}_h,
\end{align}
estimates
\eqref{eq:est_red_h_aposteriori_s} and \eqref{eq:est_red_h_aposteriori_J} become
\begin{align}
    \qquad \alpha \norm[]{L2c}{\bar{u}_h-\bar{u}}
    +\norm[]{L2o}{\bar{y}_h-\bar{y}}
    &\leq
    \frac{2}{\alpha} \norm[]{L2c}{
    (-\Delta)^{-1}[\res_w] + \res_u}^2
    +4\norm[]{L2o}{(-\Delta)^{-1}[\res_y]}^2,\label{eq:est_red_h_aposteriori_ex1}\\
    J_\alpha(\bar{u}_h,\bar{y}_h)-J_\alpha(\bar{u},\bar{y})
    &\leq
    \begin{multlined}[t]
        \frac{1}{2\alpha} \norm[]{L2c}{(-\Delta)^{-1}[\res_w] + \res_u}^2 \\
        +
        (\bar{y}_h-g^\delta,(-\Delta)^{-1}[\res_y])_{{L}^{2}(\omega_o)}
        .
        \label{eq:est_red_h_aposteriori_ex1_J}
    \end{multlined}
\end{align}

It remains to describe how the right-hand sides can be evaluated for a given
discrete approximation $(\bar u_h,\bar y_h)$.
The residual $\res_w$ can be estimated using a conventional error estimator:
Observing that $\bar{w}_h$ and $\hat{w}$ solve the discretized and continuous Poisson equation with the same right-hand side $C^*(C\bar{y}_h-g^\delta)$, we can write
\begin{equation}\label{Ares1}
    (\Apa)^{-1} \res_w=(\bar{w}_h-\hat{w})
    =((\Apa_h)^{-1}-(\Apa)^{-1})C^*(C\bar{y}_h-g^\delta).
\end{equation}
Hence, using duality-based error estimators, e.g., from \cite[Sec.~2.4]{ainsworthoden}, with $\phi=A^{-1}BB^*(\bar{w}_h-\hat{w})\in\calY$, we obtain
\begin{equation}
    \begin{aligned}
        \norm[]{U}{B^*(A^*)^{-1}\res_w}&=\norm[]{U}{B^*(\bar{w}_h-\hat{w})}
        =\langle A\phi,\bar{w}_h-\hat{w}\rangle_{\calW^*,\calW}\\
        &=\langle \phi,A^*(\bar{w}_h-\hat{w})\rangle_{\calY,\calY^*}
        =\langle \phi-\calI^\calT\phi,A^*(\bar{w}_h-\hat{w})\rangle_{\calY,\calY^*},
    \end{aligned}
\end{equation}
where we have used Galerkin orthogonality in the last equality.
Since $\Omega$ is assumed to be convex and polyhedral,
we can apply \eqref{eq:estint} to $\phi \in H^2(\Omega)$ to obtain for all $K\in\calT_h$ the estimate
\begin{equation}
    \|\phi-\calI^\calT\phi\|_{L^2(K)} +
    h_K^{1/2} \norm[]{L2pK}{v-\calI^\calT v} \leq \Cint h_K^2 |\phi|_{H^2(K)}, 
\end{equation}
Due to $H^2$-regularity, we can also apply \eqref{eq:eststab} to further estimate
$|\phi|_{H^2(\Omega)} \leq \Cstb \norm[]{L2c}{\bar{w}_h-\hat{w}}$.
From \eqref{eq:yhatwhat} and integration by parts, we thus obtain
\begin{equation}
    \begin{aligned}[t]
        \norm[]{L2c}{\bar{w}_h-\hat{w}}^2
        &= \int_\Omega\nabla (\phi-\calI^\calT\phi)\cdot\nabla(\bar{w}_h-\hat{w})\, dx\\
        &= \sum_{K\in \calT_h} \left(\int_K(\phi -\calI^\calT\phi)\res_w\, dx
    +\int_{\partial K} (\phi -\calI^\calT\phi)\nabla\bar{w}_h\cdot \nu\, ds \right)\\
    &
    \le \Cint\sum_{K\in \calT_h} \left(    h_K^2 \norm[]{L2K}{\res_w}
        + \frac12 h_K^{3/2} \norm[]{L2pK}{\llbracket\nabla \bar{w}_h\cdot \nu\rrbracket}
    \right)|\phi|_{H^2(K)}
    \\
    &\leq
    c_\calT \left(  \sum_{K\in \calT_h} h_K^4 \norm[]{L2K}{\res_w}^2
    + \frac12 h_K^{3} \norm[]{L2pK}{\llbracket\nabla \bar{w}_h\cdot \nu\rrbracket}^2 \right)^{1/2}\norm[]{L2c}{\bar{w}_h-\hat{w}},
\end{aligned}
    \end{equation}
    where $c_\calT:=\Cint\Cstb$, and
    $\llbracket\cdot\rrbracket$ denotes the jump over the element boundary $\partial K$ with normal $\nu$.
    Canceling the norm on both sides then yields
    \begin{equation}\label{eq:Hm2}
        \begin{aligned}[t]
            \norm[]{L2c}{(-\Delta)^{-1}[\res_w]}
            &=\norm[]{L2c}{\bar{w}_h-\hat{w}}\\
            &\leq
            c_\calT \left(  \sum_{K\in \calT_h} h_K^4 \norm[]{L2K}{\res_w}^2
            + \frac12 h_K^{3} \norm[]{L2pK}{\llbracket\nabla \bar{w}_h\cdot \nu\rrbracket}^2 \right)^{1/2}\\
            &=:c_\calT \ \eta_{w}.
        \end{aligned}
    \end{equation}
    Note that although $\res_w$ is globally only an element of $H^{-1}(\Omega)$, we may take its elementwise $L^2(K)$ norm, since $\bar{w}_h$ is piecewise polynomial and therefore $\Delta (\bar{w}_h\vert_K)\in L^2(K)$. In case of piecewise linear finite elements, we just have $\norm[]{L2K}{\chi_{\omega_o}(\bar{y}_h-g^\delta)}^2$ in place of $\norm[]{L2K}{\res_w}^2$.

    The term containing $\res_u$ is straightforward to evaluate as a sum of elementwise contributions.
    Analogously to \eqref{Ares1}, we have a similar representation for $\res_y$ in  \eqref{CAres3}. As in \eqref{eq:Hm2},
    we can thus estimate
    \begin{equation}\label{eq:Hm2_res3}
        \begin{aligned}[t]
            \norm[]{L2o}{\bar{y}_h-\hat{y}}
            &\leq c_\calT \left(
            \sum_{K\in \calT_h} h_K^4 \norm[]{L2K}{\res_y}^2
        +  \frac12 h_K^{3} \norm[]{L2pK}{\llbracket\nabla \bar{y}_h\cdot \nu\rrbracket}^2\right)^{1/2}\\
        &=:c_\calT \eta_{y}.
    \end{aligned}
\end{equation}
Combining \eqref{eq:est_red_h_aposteriori_ex1} and \eqref{eq:est_red_h_aposteriori_ex1_J} with \eqref{eq:Hm2}, and \eqref{eq:Hm2_res3}, we thus obtain the explicit a posteriori estimates
\begin{equation}
    \begin{aligned}
        \alpha \norm[]{L2c}{\bar{u}_h-\bar{u}}^2
        +\norm[]{L2o}{\bar{y}_h-\bar{y}}^2
        &\leq
        \frac{2}{\alpha} \norm[]{L2c}{(-\Delta)^{-1}[\res_w] + \res_u}^2
        +4\norm[]{L2o}{\bar{y}_h-\hat{y}}^2\\
        &\leq
        \frac{2}{\alpha}
        \left(c_\calT  \eta_w + \|\res_u\|_{L^2(\omega_c)}^2\right)^2
        + 4 \left(c_\calT \eta_y\right)^2,\\
        J_\alpha(\bar{u}_h,\bar{y}_h)-J_\alpha(\bar{u},\bar{y})
        &\leq
        \frac{1}{2\alpha}
        \left(c_\calT  \eta_w + \|\res_u\|_{L^2(\omega_c)}^2\right)^2
        + c_\calT \eta_y \norm[]{L2o}{\bar{y}_h-g^\delta}.
    \end{aligned}
\end{equation}
\begin{remark}
    The $L^2$ inner product term in \eqref{eq:est_red_h_aposteriori_ex1_J} could in principle lead to a negative estimate of $J_\alpha(\bar{u}_h,\bar{y}_h)-J_\alpha(\bar{u},\bar{y})$, which by \eqref{eq:accuracy_cost} would mean that no refinement is required from the point of view of cost  functional accuracy.
    However, so far we have not found a means to reasonably evaluate this term as a possibly negative inner product (approximating $(-\Delta)^{-1}$ by its discretized version would just make the term vanish) and thus to estimate it by the Cauchy--Schwarz inequality.

    Estimates~\eqref{eq:Hm2} and \eqref{eq:Hm2_res3} give bounds on quantities defined on the possibly restricted subdomains $\omega_c$ and $\omega_o$, respectively. However, the estimators are sums of contributions on the whole domain $\Omega$, and the dependence on the subdomains $\omega_c$, $\omega_o$ only enters indirectly via the definition of $\res_w$, $\bar{w}_h$, $\res_y$, and $\bar{y}_h$. Still, this makes sense, since these estimators are supposed to indicate local refinement of the finite element mesh for $\bar{w}_h$ and $\bar{y}_h$ defined on all of $\Omega$.
\end{remark}

\begin{remark}
    Related results can be found in the literature on a posteriori error estimates for optimal control problems.
    We mention \cite{kohlssiebertroesch14,liuyan01}, where $H^1$-error estimates are used in contrast to the $L^2$-estimators employed above.
    Goal-oriented error estimators of dual-weighted-residual type are investigated in, e.g.,
    \cite{beckervexler04,GriesbaumKaltenbacherVexler08,KaltenbacherKirchnerVexler11,KaltenbacherKirchnerVexler13}.
\end{remark}

\section{Banach space norm constraint}\label{sec:errest_red_Ban1}

In this section, we consider as regularization term
\begin{equation}
    \calR_\alpha=\delta_{B_{1/\alpha}^{\calU}}, \qquad\text{and hence}
    \qquad  \calR_\alpha^*=\frac{1}{\alpha}\norm[]{X}{\cdot}.
\end{equation}
This setting is of particular interest for incorporating pointwise almost everywhere bounds on $u$ via $\calU=L^\infty(\omega_c)$; see \cref{subsec:example2} below.
Let us recall that in the setting $\cal U=\calX^*$, the operator $B$ is explicitly assumed to be an adjoint operator, which is the case in the example considered in \cref{subsec:example2}.

Using the definitions of \cref{sec:spaces} and standard arguments from convex analysis, we obtain for \eqref{eq:Tikh} the optimality conditions
\begin{equation}\label{eq:optsys_B}
    \left\{\begin{aligned}
            C^*(C\bar{y}-g^\delta)+\Apa\bar{w}&=0,\\
            \bar{u}\in B_{1/\alpha}^{\calU}\quad\text{and}\quad\langle u-\bar{u}, \Bpa\bar{w}\rangle_{\calU,\calX}&\leq0 \quad \forall u\in B_{1/\alpha}^{\calU},\\
            A\bar{y}-B\bar{u}&=0.
    \end{aligned}\right.
\end{equation}
The corresponding discrete optimality conditions are
\begin{equation}\label{eq:optsys_discr_B}
    \left\{\begin{aligned}
            R_{\calY_h} \left(C^*(C\bar{y}_h-g^\delta)+ \Apa\bar{w}_h\right)&=0,\\
            \bar{u}_h\in B_{1/\alpha}^{\calU_h}\quad\text{and}\quad\langle u_h-\bar{u}_h, \Bpa\bar{w}_h\rangle_{\calU,\calX}&\leq0 \quad \forall u_h\in B_{1/\alpha}^{\calU_h},\\
            R_{\calW_h}\left(A\bar{y}_h-B\bar{u}_h\right)&=0.
    \end{aligned}\right.
\end{equation}

\subsection{Error estimates}

Setting
$\phi_\alpha(u_1,u_2)=0$,
we obtain from \cref{th:Repin} that the solution $\bar{u}$ to \eqref{eq:minred} satisfies
\begin{equation}\label{eq:est_red_B}
    \begin{aligned}
        \norm[]{G}{Ku-K\bar{u}}^2
        &\leq2\left(J_\alpha(u,Ku)-J_\alpha(\bar{u},K\bar{u})\right)\\
        &\leq\frac{2}{\alpha}\norm[]{X}{K^*g^*}
        +\norm[]{G}{Ku-g^\delta}^2+\norm[]{G}{ g^*-g^\delta}^2-\norm[]{G}{g^\delta}^2
        \\
        &=
        \frac{2}{\alpha}\norm[]{X}{K^* (g-g^\delta)}
        +2\langle u, K^* (g-g^\delta)\rangle_{\calU,\calX}
        +\norm[]{G}{Ku-g}^2
    \end{aligned}
\end{equation}
for any $u\in \calU$ and $g^*:=g^\delta -g \in \calG$ for any $g\in \calG$.
Inserting $u=\bar{u}_h$ and $g=C\bar{y}_h$ with $(\bar{u}_h,\bar{y}_h,\bar{w}_h) \in \calU_h\times \calY_h\times \calW_h$ satisfying \eqref{eq:optsys_discr_B}, we obtain
\begin{equation}\label{eq:est_red_Bh}
    \begin{aligned}[t]
        {\norm[]{G}{C\hat{y}-C\bar{y}}^2}
        &\leq2\left(J_\alpha(\bar{u}_h,\hat{y})-J_\alpha(\bar{u},\bar{y})\right)\\
        &\leq
        \frac{2}{\alpha} \norm[]{X}{\Bpa \hat{w}}
        - 2\langle \bar{u}_h, \Bpa \hat{w}\rangle_{\calU,\calX}
        +\norm[]{G}{CA^{-1}(A\bar{y}_h-B\bar{u}_h)}^2,
    \end{aligned}
\end{equation}
with $\hat{y}$ and $\hat w$ defined as in \eqref{eq:yhat} and \eqref{eq:yhatwhat}, respectively.
Note that by $\norm[]{U}{\bar{u}_h}\leq \frac{1}{\alpha}$, the term $\frac{1}{\alpha} \norm[]{X}{\Bpa \hat{w}} - \langle \bar{u}_h, \Bpa \hat{w}\rangle_{\calU,\calX}$ is indeed nonnegative.

For the first and last relation in \eqref{eq:optsys_discr_B}, we can define the residuals $\res_w$ and $\res_y$ as in \eqref{eq:res123} and, taking into account
\eqref{CAres3}--\eqref{estJ},
obtain a first a posteriori estimate.
\begin{proposition}\label{prop40}
    Let $\calR_\alpha = \delta_{B_{1/\alpha}^{\calU}}$. Then the  minimizers $(\bar{u},\bar{y})$ of \eqref{eq:Tikh} and  $(\bar{u}_h,\bar{y}_h)$ of  \eqref{eq:Tikh_h} satisfy the estimates
    \begin{align}
        {\norm[]{G}{C\bar{y}_h-C\bar{y}}^2}
        &\leq
        \frac{4}{\alpha} \norm[]{X}{\Bpa \hat{w}}
        - 4\langle \bar{u}_h, \Bpa \hat{w}\rangle_{\calU,\calX}
        +4\norm[]{G}{CA^{-1}\res_y}^2,\\
        J_\alpha(\bar{u}_h,\bar{y}_h)-J_\alpha(\bar{u},\bar{y})
        &\leq
        \frac{1}{\alpha} \norm[]{X}{\Bpa \hat{w}}
        - \langle \bar{u}_h, \Bpa \hat{w}\rangle_{\calU,\calX}
        +\norm[]{G}{CA^{-1}\res_y}\norm[]{G}{C\bar{y}_h-g^\delta}.
    \end{align}
    with $\res_y$ as in \eqref{eq:res123}.
\end{proposition}

If a duality mapping ${\calJ}^{\calX}(x)\in \partial\norm[]{X}{\cdot}(x)$ exists (e.g., if $\calU=\calX^*$), we can also define a residual for the second relation in \eqref{eq:Tikh_h} by
\begin{equation}\label{eq:res2til}
    {\res}_u:=\alpha \bar{u}_h - {\calJ}^{\calX}(\Bpa \bar{w}_h).
\end{equation}
From $\langle{\calJ}^{\calX} (\Bpa \bar{w}_h),\Bpa \bar w_h\rangle_{\calU,\calX}=\norm[]{X}{\Bpa \bar w_h}$ it follows that
$\langle {\res}_u,\Bpa \bar w_h\rangle_{\calU,\calX}\le0$.
Then we can estimate
\begin{equation}
    \norm[]{X}{\Bpa \hat{w}}
    + \alpha\langle \bar{u}_h, \Bpa \hat{w}\rangle_{\calU,\calX}
    =\langle -{\res}_u - ({\calJ}^{\calX} (\Bpa \bar{w}_h) -{\calJ}^{\calX}( \Bpa \hat w)) ,\Bpa \hat{w}\rangle_{\calU,\calX}.
\end{equation}
By construction we have that
\begin{equation}
    \langle {\calJ}^{\calX} (\Bpa \bar{w}_h) -{\calJ}^{\calX}( \Bpa \hat w) ,\Bpa \bar w_h\rangle_{\calU,\calX}
    = \|\Bpa \bar{w}_h\|_\calX - \langle {\calJ}^{\calX}( \Bpa \hat w) ,\Bpa \bar w_h\rangle_{\calU,\calX} \ge 0
\end{equation}
Hence it follows that
\begin{equation}
\langle {\calJ}^{\calX} (\Bpa \hat w) -{\calJ}^{\calX}( \Bpa \bar w_h)) ,\Bpa \hat{w}\rangle_{\calU,\calX}
    \le \langle {\calJ}^{\calX} (\Bpa \hat w) -{\calJ}^{\calX}( \Bpa \bar w_h)) ,\Bpa (\hat{w}-w_h)\rangle_{\calU,\calX}.
\end{equation}
Introducing the symmetric Bregman distance of $\|\cdot\|_\calU$ defined as
\begin{equation}\label{dsymmnx}
\dsymmnx(\Bpa \hat w, \Bpa \bar w_h) :=\langle {\calJ}^{\calX} (\Bpa \hat w) -{\calJ}^{\calX}( \Bpa \bar w_h)) ,\Bpa (\hat{w}-w_h)\rangle_{\calU,\calX},
\end{equation}
we obtain the estimate
\begin{equation}
    \norm[]{X}{\Bpa \hat{w}}
    - \alpha\langle \bar{u}_h, \Bpa \hat{w}\rangle_{\calU,\calX}
    \le \langle -{\res}_u,\Bpa \hat{w}\rangle_{\calU,\calX} + \dsymmnx(\Bpa \hat w, \Bpa \bar w_h).
\end{equation}

Using \eqref{CAres3} in \eqref{eq:est_red_Bh} together with
the definitions of $ \res_u$ and $\dsymmnx$ yields the following estimates.
\begin{proposition}\label{prop41}
    Let $\calU=\calX^*$ and $\calR_\alpha = \delta_{B_{1/\alpha}^{\calU}}$. Then the  minimizers $(\bar{u},\bar{y})$ of \eqref{eq:Tikh} and  $(\bar{u}_h,\bar{y}_h)$ of  \eqref{eq:Tikh_h} satisfy the estimates
    \begin{align}
        \norm[]{G}{C\bar{y}_h-C\bar{y}}^2
        &\leq
        \frac4\alpha \langle -{\res}_u,\Bpa \hat{w}\rangle_{\calU,\calX} + \frac4\alpha\dsymmnx(\Bpa \hat w)             +4\norm[]{G}{CA^{-1}\res_y}^2,\label{eq:est_red_Bh_s}\\
        \qquad J_\alpha(\bar{u}_h,\bar{y}_h)-J_\alpha(\bar{u},\bar{y})
        &\leq
        \frac1\alpha \langle -{\res}_u,\Bpa \hat{w}\rangle_{\calU,\calX} + \frac1\alpha\dsymmnx(\Bpa \hat w, \Bpa \bar w_h)           \label{eq:est_red_Bh_s_J}\\
        \MoveEqLeft[-1]  +\norm[]{G}{CA^{-1}\res_y}\norm[]{G}{C\bar{y}_h-g^\delta},
    \end{align}
    with
    $\res_y$ as in \eqref{eq:res123} and $\res_u$ as in \eqref{eq:res2til}.
\end{proposition}
Let us remark that due to \eqref{eq:res123}, the unknown $\hat w$ can replaced  by $\bar w_h -(A^{*})^{-1}\res_w$.
Hence, the components of the error estimate are fully available in numerical implementations,
as we will show in more detail in \cref{subsec:example2}.

If a variational discretization, i.e., $\calU_h=\calU$, is used, then
from \eqref{eq:optsys_discr_B} we obtain $\Bpa \bar w_h \in \partial \delta_{B_{1/\alpha}^{\calU}}(\bar u_h)$,
which is equivalent to $\bar u_h \in \partial \|\cdot\|_\calX(\Bpa\bar w_h)$.
This implies that $\res_u=0$, and hence \eqref{eq:est_red_Bh_s} and \eqref{eq:est_red_Bh_s_J} reduce to
\begin{align}
    \norm[]{G}{C\bar{y}_h-C\bar{y}}^2
    &\leq \frac4\alpha\dsymmnx(\Bpa \hat w, \Bpa \bar w_h)+4\norm[]{G}{CA^{-1}\res_y}^2,\label{eq:est_red_Bh_s_var}\\
    J_\alpha(\bar{u}_h,\bar{y}_h)-J_\alpha(\bar{u},\bar{y})
    &\leq \frac1\alpha\dsymmnx(\Bpa \hat w, \Bpa \bar w_h)+\norm[]{G}{CA^{-1}\res_y}\norm[]{G}{C\bar{y}_h-g^\delta}.\label{eq:est_red_Bh_s_var_J}
\end{align}

\subsection{Application to inverse source problem}\label{subsec:example2}

We now apply the estimate from \cref{prop41} to the model problem \eqref{eq:ex1} for the case of Ivanov regularization. In this case, we have $\calU=L^\infty(\omega_c)$ and $\calX=L^1(\omega_c)$, i.e., $\calU=\calX^*$, and hence the duality mapping is given by
\begin{equation}
    {\calJ}^{\calX} (\Bpa w)=\sign(w\vert_{\omega_c}).
\end{equation}
As before, we take $\calY= H_0^1(\Omega)=\calW$. The Ivanov problem is then given by
\begin{equation}\label{eq:ex1bmin}
    \left\{\begin{aligned}
            &\min_{y,u} \frac12 \norm[]{L2o}{y-g^\delta}^2
            \quad \text{ s.t. }\quad|u(x)|\leq\frac{1}{\alpha}\quad\text{for a.e. } x\in\omega_c\\
            &\text{and }-\Delta y = \chi_{\omega_c} u,\quad y|_{\partial\Omega}=0.
    \end{aligned}\right.
\end{equation}
The residuals used in \cref{prop41} are now given by
\begin{align}
    \res_w&:=\chi_{\omega_o}(\bar{y}_h-g^\delta)-\Delta\bar{w}_h,\\
    {\res}_u&:=\alpha \bar{u}_h-\sign(\bar w_h\vert_{\omega_c}),\\
    \res_y&:=-\Delta\bar{y}_h-\chi_{\omega_c}\bar{u}_h.
\end{align}
We will consider the case of variational discretization for simplicity,
where we can make use of the estimate \eqref{eq:est_red_Bh_s_var}.
Since the term containing $\res_y$ in \eqref{eq:est_red_Bh_s_var} can be estimated by \eqref{eq:Hm2_res3}, it
only remains to consider the term containing $\dsymmnx$, which in this setting can be estimated by
\begin{equation}
    \dsymmnx(\Bpa \hat w, \Bpa \bar w_h)=
    \frac{1}{\alpha}\scalprod{LinftycL1c}{\sign(\hat{w})-\sign(\bar{w}_h)}{\hat{w}-\bar{w}_h}
    \leq \frac{2}{\alpha}\norm[]{L1c}{\hat{w}-\bar{w}_h}.
\end{equation}
(Note that we cannot expect smallness of $\norm[]{Linftyc}{\sign(\hat{w})-\sign(\bar{w}_h)}$ directly, since continuity of the $\sign$ operator cannot be quantified on $\calX=L^1(\omega_c)$.)

In order to estimate the $\m{L1c}$-norm of $\hat{w}-\bar{w}_h$, we introduce
\begin{equation}
    z := (-\Delta)^{-1}[\chi_{\omega_c} \sign(\hat{w}-\bar{w}_h) ]\ \in\  W^{2,p}(\Omega)\cap H^1_0(\Omega) \quad \forall p<\infty.
\end{equation}
We assume from here on that $\Omega\subset \R^2$ is polygonal with interior angles of at most $\frac\pi2$.
In this case, we obtain from \cite[Thm.~1]{dipliniotemam2015} that
\begin{equation}\label{eq:schatzwahlbin}
    \|z\|_{W^{2,p}(\Omega)} \le c_S \, p \, \norm[]{Linf}{\sign(\hat{w}-\bar{w}_h)} \leq c_S\,p
\end{equation}
holds for all $p\ge2$ with a constant $c_S>0$ independent of $p$.
In case that $\Omega$ does not allow for such a regularity result, and \eqref{eq:schatzwahlbin} only holds for $p=2$,
we can use the $L^2$-error estimate of \cref{subsec:example1}.

Let $\calI^\calT z$ be the piecewise linear interpolant of $z$.
Then we have from \cite[Thm 3.1.6]{Ciarlet78} together with \eqref{eq:schatzwahlbin} for all $p>d$ the estimate
\begin{equation}\label{eq42_interp}
    \norm[]{LinftyK}{z - \calI^\calT z} +  \norm[]{LinftyFF}{z - \calI^\calT z} \le c_I \, h_K^{2-d/p}  \norm[]{W2pK}{z} \leq c_Ic_S\,p\, h_K^{2-d/p}
\end{equation}
with a  constant $c_I>0$ depending only on the chosen finite element family.
Using the definition of $z$, we obtain
\begin{equation}
    \begin{aligned}
        \norm[]{L1c}{\hat{w}-\bar{w}_h}
        &=(\nabla z,\nabla (\hat{w}-\bar{w}_h))_{\m{L2}}\\
        &=(\nabla (z-\calI^\calT z),\nabla (\hat{w}-\bar{w}_h))_{\m{L2}}
        \\
        &=-( z-\calI^\calT z,\bar{y}_h-g^\delta)_{\m{L2o}} -(\nabla (z-\calI^\calT z),\nabla\bar{w}_h)_{\m{L2}},
    \end{aligned}
\end{equation}
where
we have used Galerkin orthogonality and the fact that the interpolation operator $\calI^\calT:{C}_b(\overline{\Omega})\to\calY_h$ can indeed be applied to
$z\in W^{2,p}(\Omega)\cap H_0^1(\Omega)\hookrightarrow {C}_b(\overline{\Omega})$.
Here and below, $(\cdot,\cdot)_{L^2}$ denotes the $L^2$ inner product.
Now we integrate by parts on each element to obtain
\begin{equation}\label{etaw_normcons}
    \begin{aligned}[t]
        \norm[]{L1c}{\hat{w}-\bar{w}_h}
        &=-\sum_{K\in \calT_h}
        \Bigl((z-\calI^\calT z,
            -\Delta \bar{w}_h+\chi_{\omega_o}(\bar{y}_h-g^\delta))_{\m{L2K}}\\
            \MoveEqLeft[-1] -\int_{\partial K} \nabla \bar{w}_h\cdot \nu (z-\calI^\calT z)\,ds
        \Bigr)\\
        &\leq \sum_{K\in \calT_h} \Big(\norm[]{LinftyK}{z-\calI^\calT z}
        \ \norm[]{L1K}{-\Delta \bar{w}_h+\chi_{\omega_o}(\bar{y}_h-g^\delta)}
        \\
        \MoveEqLeft[-1]
        +\norm[]{LinftypK}{z-\calI^\calT z}
    \norm[]{L1pK}{\llbracket\nabla \bar{w}_h\cdot \nu\rrbracket} \Big)
    \\
    &\leq c_Ic_S
    \sum_{K\in \calT_h}
    p_K \,
    h_K^{2-\frac d{p_K}}
    \ \left(\norm[]{L1K}{-\Delta \bar{w}_h+\chi_{\omega_o}(\bar{y}_h-g^\delta)}
        +
        \norm[]{L1pK}{\llbracket\nabla \bar{w}_h\cdot \nu\rrbracket}
    \right)
\end{aligned}
\end{equation}
where we have used \eqref{eq42_interp} with $p_K\ge d$ individually for each element $K\in\calT_h$.
(As in \eqref{eq:Hm2}, the term $\Delta \bar{w}_h$ vanishes in case of piecewise linear finite elements.)
Choosing now $p_K \sim d\,|\log(h_K)|$ yields
\begin{equation}
    \norm[]{L1c}{\hat{w}-\bar{w}_h}\le c_Ic_S
    \sum_{K\in \calT_h}
    |\log h_K| \,
    h_K^{2}
    \left(\norm[]{L1K}{\rho_w}
        +
        \norm[]{L1pK}{\llbracket\nabla \bar{w}_h\cdot \nu\rrbracket}
    \right)=:c_\calT \eta_w.
\end{equation}
With the help of this residual-based error estimate and of \eqref{eq:Hm2_res3}, the error estimates \eqref{eq:est_red_Bh_s_var} and \eqref{eq:est_red_Bh_s_var_J} can be computed.

\section{Banach space norm regularization}\label{sec:errest_red_Ban2}

In this section, we consider as regularization term
\begin{equation}
    \calR_\alpha=\alpha \norm[]{U}{\cdot}, \qquad\text{and hence}\qquad
    \calR_\alpha^*=\delta_{B_\alpha^\calX}(\cdot).
\end{equation}
This setting is of particular interest for promoting sparsity of $u$ via $\calU=\mathcal{M}(\Omega)$; see \cref{subsec:example3}.
Again, in case $\calU=\calX^*$ we explicitly assume that $B$ is an adjoint operator.

As above, we obtain for \eqref{eq:Tikh} the  optimality conditions
\begin{equation}\label{eq:optsys_C}
    \left\{\begin{aligned}
            C^*(C\bar{y}-g^\delta)+\Apa\bar{w}&=0,\\
            \Bpa\bar{w}\in B_\alpha^{\calX}\quad\text{and}\quad
            \langle\bar{u},u^*- \Bpa\bar{w}\rangle_{\calU,\calX}&\leq0 \quad \forall u^*\in B_\alpha^{\calX},
            \\
            A\bar{y}-B\bar{u}&=0.
    \end{aligned}\right.
\end{equation}
We again consider a discretization of this system.
In the following, let $(\bar y_h,\bar u_h,\bar w_h)$ be a discrete approximation of the solution of \eqref{eq:optsys_C} given by
\begin{equation}\label{eq:optsys_discr_C}
    \left\{\begin{aligned}
            R_{\calY_h} \left(C^*(C\bar{y}_h-g^\delta)+ \Apa\bar{w}_h\right)&=0,\\
            R_{\calW_h}\left(A\bar{y}_h-B\bar{u}_h\right)&=0,
    \end{aligned}\right.
\end{equation}
together with a discretization of the second relation of \eqref{eq:optsys_C}, which however is intimately
linked to the choice of the space $\calU_h$ and the discrete approximation
of $\Bpa w\in B_\alpha^{\calX}$.
We refer to \cref{subsec:example3} concerning details for the specific choice $\calU=\mathcal{M}(\Omega)$.

\subsection{Error estimates}\label{sec:errest_red_Ban2_est}

Setting again $\phi_\alpha(u_1,u_2)=0$, we obtain from \cref{th:Repin} that the solution $\bar{u}$ to \eqref{eq:minred} satisfies
\begin{equation}\label{eq:est_red_C}
    \begin{aligned}
        {\norm[]{G}{Ku-K\bar{u}}^2}&\leq2\left(J_\alpha(u,Ku)-J_\alpha(\bar{u},K\bar{u})\right)\\
                                 &\leq 2\alpha \norm[]{U}{u}
        +\norm[]{G}{Ku-g^\delta}^2+\norm[]{G}{ g^*-g^\delta}^2-\norm[]{G}{g^\delta}^2
        \\
        &= 2\alpha \norm[]{U}{u}
        +2\langle u, K^* (g-g^\delta)\rangle_{\calU,\calX}
        +\norm[]{G}{Ku-g}^2.
    \end{aligned}
\end{equation}
for any $u\in \calU$ and $g^*:=g^\delta -g \in \calG$ for any $g\in \calG$ such that  $K^*g^*\in B_{\alpha}^{\calX}$.

Similarly as before, we set $u=\bar u_h$.
However, the choice $g=C\bar{y}_h$ is not possible, as $K^*(C\bar{y}_h-g^\delta)\notin B_{\alpha}^{\calX}$ in general.
Hence, we introduce a scaling factor $\kappa>0$ such that
$g- g^\delta = \kappa(C\bar{y}_h-y^\delta)$
satisfies
\begin{equation}
    \kappa K^*(C\bar{y}_h- g^\delta) = \kappa B^*(A^*)^{-1}C^*(C\bar{y}_h-g^\delta)
    = \kappa \Bpa \hat w \in B_{\alpha}^{\calX}
\end{equation}
with $\hat w$ as in \eqref{eq:yhatwhat}.
It thus suffices to choose
\begin{equation}\label{eq:kappa}
    \kappa=\min\left\{\frac{\alpha}{\norm[]{X}{\Bpa \hat{w}}},\ 1\right\}.
\end{equation}
The estimation of $\kappa$ will be discussed below; see \eqref{eq:est_kappam1_wh} and \eqref{eq:etakappa}.
Inserting $u=\bar{u}_h$ and $g=\kappa C\bar{y}_h+(1-\kappa)g^\delta$ with $(\bar{u}_h,\bar{y}_h,\bar{w}_h) \in \calU_h\times \calY_h\times \calW_h$ satisfying \eqref{eq:optsys_discr_C}, we obtain
\begin{equation}\label{eq:est_red_Ch}
    \begin{aligned}
        \norm[]{G}{C\hat{y}-C\bar{y}}^2
        &\leq2\left(J_\alpha(\bar{u}_h,\hat{y})-J_\alpha(\bar{u},\bar{y})\right)\\
        &\leq
        2\alpha \norm[]{U}{\bar{u}_h}
        - 2\langle \bar{u}_h, \kappa\Bpa \hat{w}\rangle_{\calU,\calX}
        +\norm[]{G}{CA^{-1}(A\bar{y}_h-B\bar{u}_h)+(\kappa-1)(C\bar{y}_h-g^\delta)}^2\\
        & =
        2(\alpha \norm[]{U}{\bar{u}_h}
            - \langle \bar{u}_h, \Bpa \bar w_h\rangle_{\calU,\calX}
            + \kappa\langle \bar{u}_h, \Bpa(\bar w_h -\hat{w})\rangle_{\calU,\calX}\\
        \MoveEqLeft[-1]        + (1-\kappa) \langle \bar{u}_h, \Bpa \bar w_h\rangle_{\calU,\calX})
        +\norm[]{G}{CA^{-1}\res_y+(\kappa-1)(C\bar{y}_h-g^\delta)}^2.
    \end{aligned}
\end{equation}
Note that by dual feasibility of $\kappa\hat{w}$, the term $\alpha \norm[]{U}{\bar{u}_h}-\langle \bar{u}_h, \kappa\Bpa \hat{w}\rangle_{\calU,\calX}$ is nonnegative.
Estimating again the terms on the right-hand side using \eqref{CAres3} and \eqref{abcd} with $\sigma=4$ and $\gamma=2$, we obtain the following a posteriori estimate.
\begin{proposition}\label{prop:est_normpenalty}
    Let $\calR_\alpha = \alpha\norm[]{U}{\cdot}$. Then the  minimizers $(\bar{u},\bar{y})$ of \eqref{eq:Tikh} and  $(\bar{u}_h,\bar{y}_h)$ of  \eqref{eq:Tikh_h} satisfy the estimate
    \begin{align}
        \norm[]{G}{C\bar{y}_h-C\bar{y}}^2
        &\leq
        4(\alpha \norm[]{U}{\bar{u}_h}
        - \langle \bar{u}_h, \Bpa \bar w_h\rangle_{\calU,\calX})
        + 4\kappa\langle \bar{u}_h, \Bpa(\bar w_h -\hat{w})\rangle_{\calU,\calX}\label{eq:est_red_Ch_1}\\
        \MoveEqLeft[-1] + 4(1-\kappa) \langle \bar{u}_h, \Bpa \bar w_h\rangle_{\calU,\calX}
        +4\norm[]{G}{CA^{-1}\res_y+(\kappa-1)(C\bar{y}_h-g^\delta)}^2,\\
        \qquad J_\alpha(\bar{u}_h,\hat{y})-J_\alpha(\bar{u},\bar{y})
        &\leq
        (\alpha \norm[]{U}{\bar{u}_h}
        - \langle \bar{u}_h, \Bpa \bar w_h\rangle_{\calU,\calX})
        + \kappa\langle \bar{u}_h, \Bpa(\bar w_h -\hat{w})\rangle_{\calU,\calX}\label{eq:est_red_Ch_1_J}\\
        \MoveEqLeft[-1] + (1-\kappa) \langle \bar{u}_h, \Bpa \bar w_h\rangle_{\calU,\calX}
        \\
        \MoveEqLeft[-1]+\norm[]{G}{CA^{-1}\res_y+(\kappa-1)(C\bar{y}_h-g^\delta)}\norm[]{G}{C\bar{y}_h-g^\delta},
    \end{align}
    with $\res_y$ as in \eqref{eq:res123} and $\kappa$ satisfying \eqref{eq:kappa}.
\end{proposition}
If a duality mapping ${\calJ}^\calU(u)\in\partial\norm[]{U}{\cdot}(u)$ exists (e.g., if $\calX=\calU^*$),
we could again define a residual for the discrete version of the second relation in \eqref{eq:optsys_C} via
${\res}_u:=\alpha{\calJ}^\calU(\bar{u}_h)- \Bpa\bar{w}_h$
and proceed similarly as in \cref{sec:errest_red_Ban2_est}. Since this will not be the case in the example below, we do not do so here.

The quantity $1-\kappa$ can be estimated by
\begin{equation}\label{eq:est_kappam1_wh}
    1-\kappa \leq
    \max \left(1- \frac{\alpha}{\norm[]{X}{\Bpa \hat{w}}}, \ 0\right)
    \le
    \max \left(\frac{\norm[]{X}{\Bpa \bar w_h}-\alpha+\norm[]{X}{\Bpa (\bar w_h-\hat{w})}}{\norm[]{X}{\Bpa \bar w_h}+\norm[]{X}{\Bpa (\bar w_h-\hat{w})}}, \ 0\right).
\end{equation}
This bound can be written in terms of the residual $\res_w$ as
\begin{equation}\label{estkappa}
    1-\kappa \le
    \max \left(\frac{\norm[]{X}{\Bpa \bar w_h}-\alpha+\norm[]{X}{\Bpa(\Apa)^{-1}\res_w}}{\norm[]{X}{\Bpa \bar w_h}+\norm[]{X}{\Bpa(\Apa)^{-1}\res_w}}, \ 0\right),
\end{equation}
which implies that the quantity $1-\kappa$ is a combination of the violation of the  dual constraint $\norm[]{X}{\Bpa \bar w_h}\le \alpha$
and the residual $\res_w$. Thus we can expect $1-\kappa$ to be small for a
sufficiently fine discretization.
We refer to \cite{RoeschWachsmuth} for a  related error estimate for state-constrained optimal control problems.

\subsection{Application to inverse source problem}\label{subsec:example3}

We now apply the estimate from \cref{prop:est_normpenalty} to the model problem \eqref{eq:ex1} for the case of sparsity regularization. In this case, we have
$\calU=\mathcal{M}(\omega_c)$ and $\calX={C}_b(\omega_c)$, i.e., $\calU=\calX^*$. Due to the low regularity of the source term, we here set $\calY= W_0^{1,q'}(\Omega)$ and $\calW= W_0^{1,q}(\Omega)$, where $q'=\frac{q}{q-1}$ with $n<q\leq\frac{2n}{n-2}$ to guarantee $W_0^{1,q'}(\Omega)\subseteq L^2(\Omega)$.
The operator $B:\mathcal{M}(\omega_c) \to W^{-1,q'}(\Omega)$ is defined as
\begin{equation}
    \langle Bu, v\rangle_{W^{-1,q'},W_0^{1,q}} = \int_{\omega_c} v\,\mathrm{d} u,
\end{equation}
with $\Bpa w = w\vert_{\omega_c}$.
The Tikhonov problem is then given by
\begin{equation}\label{eq:ex1cmin}
    \left\{\begin{aligned}
            &\min_{y,u} \frac12 \norm[]{L2o}{y-g^\delta}^2 +\alpha\norm[]{Mc}{u}\\
            &\mbox{s.t. }-\Delta y = \chi_{\omega_c} u,\quad y|_{\partial\Omega}=0.
    \end{aligned}\right.
\end{equation}
From \cite{Clason:2012}, we have existence of a minimizer $\bar u\in\mathcal{M}(\omega_c)$ as well as an optimal state $\bar y\in W_0^{1,q'}(\Omega)$ and an adjoint state $\bar w\in W_0^{1,q}(\Omega)$ satisfying the optimality conditions
\begin{equation}\label{eq:optsys_C_ex1}
    \left\{\begin{aligned}
            &-\Delta\bar{w}+\chi_{\omega_o}(\bar{y}-g^\delta)=0,\quad \bar{y}|_{\partial\Omega}=0\\
            &\norm[]{Cc}{\bar{w}}\leq\alpha \ \text{ and } \
            \scalprod{McCc}{\bar{u}_h}{\tilde w- \bar{w}}\leq0
            \quad \forall \norm[]{Cc}{\tilde{w}}\leq\alpha,\\
            &-\Delta \bar{y}-\chi_{\omega_c}\bar{u}=0,\quad \bar{y}|_{\partial\Omega}=0.
    \end{aligned}\right.
\end{equation}
As $\Omega$ is convex and polyhedral, we can employ $H^2$-regularity results.
We take here as well $\calY_h\subset\calY$ and $\calW_h\subset\calW$ as piecewise linear finite elements, and thus the residuals in the first and third relation are once more given by
\begin{align}\label{eq:res123_ex1c}
    \res_w&=\chi_{\omega_o}(\bar{y}_h-g^\delta)-\Delta\bar{w}_h,\\
    \res_y&=-\Delta\bar{y}_h-\chi_{\omega_c}\bar{u}_h.
\end{align}
We again use a variational discretization $\calU_h=\calU$. It was shown in \cite{CCK12} that the corresponding semi-discretization of \eqref{eq:ex1cmin} admits a unique minimizer of the form $\bar{u}_h=\sum_{j=1}^{N_c} u_j \delta_{x_j}$, where $\delta_x$ denotes the Dirac measure concentrated on $x\in\Omega$ and $\{x_j\}_{j=1}^{N_c}$ are the interior vertices of $\calT_h$ lying in $\omega_c$. Hence, we have that
\begin{equation}
    \alpha\norm[]{Mc}{\bar{u}_h}=\scalprod{McCc}{\bar{u}_h}{\bar w_h},
\end{equation}
so that the first term on the right-hand sides of \eqref{eq:est_red_Ch_1} and \eqref{eq:est_red_Ch_1_J} vanish. Furthermore, from \cite{CCK12} we have that
\begin{equation}
    \scalprod{McCc}{\bar {u}_h}{w_h} = \sum_{j=1}^{N_c}u_j w_j,
\end{equation}
for any $w_h = \sum_{j=1}^{N_c} w_j e_j$, where $e_j$ is the piecewise linear finite element basis functions corresponding to the vertex $x_j$.

To estimate the term $A^{-1}\res_y$, we use the residual error estimator for Dirac measure data from  \cite{arayabr06} (note that here $\res_y\vert_K\notin L^2(K)$):
There exists a constant $c_2>0$ independent of $h$ such that
\begin{equation}\label{eq:etay2}
    \norm[]{L2o}{\bar{y}_h-\hat{y}}
    \leq
    c_2  \left( \sum_{K\in \calT_h} h_K^{3} \norm[]{L2pK}{\llbracket\nabla \bar{y}_h\cdot \nu\rrbracket}^2\right)^{1/2}
    =:c_2\eta_y.
\end{equation}

The term $\scalprod{McCc}{\bar{u}_h}{\bar w_h - \hat w}$ from the right-hand side of \eqref{eq:est_red_Ch_1} can be estimated as
\begin{equation}\label{eq:etaw}
    \begin{aligned}[t]
        \allowdisplaybreaks
        \scalprod{McCc}{\bar{u}_h}{\bar w_h - \hat w}
        &= \sum_{j=1}^{N_h} u_j (\bar{w}_h(x_j)-\hat{w}(x_j))
        = \sum_{j=1}^{N_h} u_j (\bar{w}_h(x_j)-\calI^\calT\hat{w}(x_j))\\
        &= \int_\Omega \nabla \bar{y}_h  \nabla(\bar{w}_h-\calI^\calT\hat{w})\,\mathrm{d}x
        = \int_\Omega \nabla \bar{y}_h  \nabla(\hat{w}-\calI^\calT\hat{w}) \,\mathrm{d}x \\
        &= \sum_{K\in \calT_h} \int_{\partial K} \nabla \bar{y}_h\cdot \nu (\hat{w}-\calI^\calT\hat{w})\,\mathrm{d}s\\
        &\leq c_2\Cint \left( \sum_{K\in \calT_h} h_K^{3} \norm[]{L2pK}{\llbracket\nabla \bar{y}_h\cdot \nu\rrbracket}^2 \right)^{1/2}
        |\hat{w}|_{H^2(\Omega)}\\
        &\leq c_2\Cint\Cstb \left( \sum_{K\in \calT_h} h_K^{3} \norm[]{L2pK}{\llbracket\nabla \bar{y}_h\cdot \nu\rrbracket}^2\right)^{1/2}
        \norm[]{L2o}{\bar{y}_h-g^\delta}\\
        &=:c_3\eta_w,
    \end{aligned}
\end{equation}
where we have used the definition of $\bar{y}_h$ in the third equality, the definition of $\bar{w}_h$ and $\hat{w}$ in the fourth equality, and elementwise integration by parts, elementwise linearity of $\bar{y}_h$ in the fifth equality, as well as \eqref{eq:estint} and \eqref{eq:eststab}.

In order to estimate $1-\kappa$, we apply the $L^\infty(\Omega)$ residual error estimator of \cite{nssv06}; see also \cite{RoeschWachsmuth},
which is valid even for nonconvex polyhedral domains.
It was proven in \cite{nssv06} that there exists a constant
$c>0$ depending on $\Omega$ and the shape regularity of the triangulation such that
\begin{equation}\label{eq:est_w_linfty}
    \begin{aligned}[t]
        \norm[]{Linftyc}{\bar w_h - \hat w}&\le c|\log h_\mathrm{min}|^2
        \max_{K\in \calT_h} \Big( h_K^2 \norm[]{LinftyK}{{-}\Delta\bar w_h+\chi_{\omega_o}(\bar y_h-g^\delta)}\\
        \MoveEqLeft[-11] + h_K \norm[]{LinftypK}{\llbracket\nabla \bar{w}_h\cdot \nu\rrbracket}\Big) =: c\, \eta_{w}^\infty,
    \end{aligned}
\end{equation}
where $h_\mathrm{min}:=\min_{K\in \calT_h}h_K$. Inserting this into \eqref{eq:est_kappam1_wh}, we obtain
\begin{equation}\label{eq:etakappa}
    1-\kappa \leq
    \max \left(\frac{\norm[]{X}{\Bpa \bar w_h}-\alpha+c\eta_w^\infty}{\norm[]{X}{\Bpa \bar w_h}+c\eta_w^\infty}, \ 0\right)=:\eta_\kappa.
\end{equation}

Collecting all the results, we obtain from \cref{prop:est_normpenalty} the a posteriori estimates
\begin{align}
    \norm[]{L2o}{\bar{y}_h-\bar{y}}^2
    &\le 4 c_3\eta_w
    +4\eta_\kappa\scalprod{McCc}{\bar{u}_h}{\bar w_h}
    \label{etaex3} \\
    \MoveEqLeft[-1]+4 \Bigl(c_2\eta_y +\eta_\kappa\norm[]{L2o}{\bar{y}_h-g^\delta}\Bigr)^2,\\
    J_\alpha(\bar{u}_h,\bar{y}_h)-J_\alpha(\bar{u},\bar{y})
    &\le c_3\eta_w
    +\eta_\kappa\scalprod{McCc}{\bar{u}_h}{\bar w_h}\label{etaex3_J}\\
    \MoveEqLeft[-1]+4 \Bigl(c_2  \eta_y +\eta_\kappa\norm[]{L2o}{\bar{y}_h-g^\delta}\Bigr)\norm[]{L2o}{\bar{y}_h-g^\delta}.
\end{align}

\begin{remark}
    A posteriori estimators for a state-constrained control problem can also be found in \cite{roeschssiebertsteinig}.
    This control problem is related to the dual problem to \eqref{eq:ex1cmin}, which takes the form of a state-constrained problem without a discrepancy term. (Conversely, the dual to the problem in \cite{roeschssiebertsteinig} involves a Huber norm in place of the measure-space norm in \eqref{eq:ex1cmin}.)
    Furthermore,  in \cite{roeschssiebertsteinig} the state constraint is penalized, which manifests in an additional $L^2$ penalty in the dual problem.
    The resulting error estimator then gives combined bounds on the   regularization and the discretization error.
\end{remark}

\section{Numerical example}\label{sec:NumTests}

We illustrate our error estimators with numerical results for the example from \cref{subsec:example3}. In order to have available an exact analytical solution, we use the example from \cite[Section 8.1]{Pieper:2013}: Setting $\Omega=\omega_c=\omega_o=B_1(0)\subseteq \R^2$, we have that $-\Delta y^\dagger = u^\dagger$ for
\begin{equation}
    y^\dagger(x)=-\frac{1}{2\pi}\ln\left(\max\left\{\rho,|x|_2\right\}\right), \qquad
    u^\dagger=-\frac{1}{2\pi\rho}\mathcal{H}^1\vert_{\partial B_\rho(0)},
\end{equation}
where $\rho\in(0,1)$ is arbitrary and $\mathcal{H}^1$ denotes the one-dimensional Hausdorff measure. Furthermore, $\bar u = u_{\alpha}^\delta = u^\dagger$ is the  minimizer of \eqref{eq:ex1cmin} for given $\alpha>0$ if the data is chosen as
\begin{equation}
    g^\delta(x)=-\frac{1}{2\pi}\ln\left(\max\left\{\rho,|x|_2\right\}\right)+\alpha\phi\left(|x|_2\right)
\end{equation}
with
\begin{equation}
    \phi(r)=\begin{cases}
        \frac{6(3r-2\rho)}{\rho^3} &\text{ for }r<\rho\\
        \frac{6(3r^2-2r\rho-2r+\rho)}{(\rho-1)^3r} &\text{ for }r\geq\rho.
    \end{cases}
\end{equation}
In the following, we set $\rho=0.5$ and $\alpha=10^{-2}$ unless specified otherwise. The corresponding discrete approximations $\bar u_h$ are computed using the approach from \cite{CCK12}.

We first illustrate \cref{prop:est_normpenalty} by comparing in \cref{fig:funcerr-pieper-vexler} the errors in residual and functional value to the terms in \eqref{etaex3} and \eqref{etaex3_J} for a sequence of adaptively refined meshes for uniform refinement (\cref{fig:funcerr-pieper-vexler:uniform}) as well as for adaptive refinement using the procedure described in \cite{RoeschWachsmuth} (\cref{fig:funcerr-pieper-vexler:adaptive}).
We also show to the rate $\mathcal{O}(h^2)$, which up to a logarithmic factor is known to hold for the residual and Tikhonov functional error; see \cite[Thm.~6.2]{Pieper:2013}. This rate also seems to be satisfied for our estimator.
\pgfplotsset{cycle list/Dark2-6}
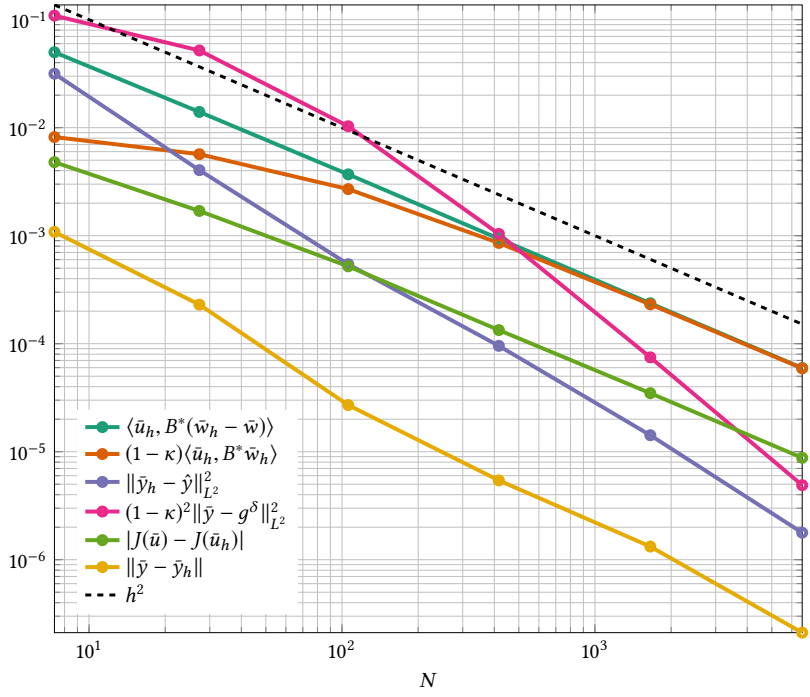
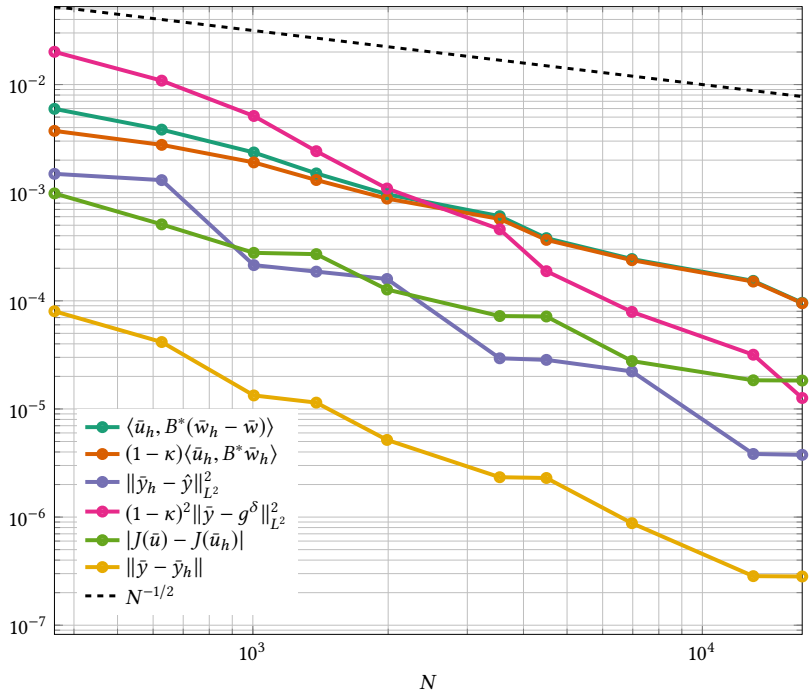
\begin{figure}
    \begin{subfigure}[t]{\textwidth}
    \centering
        \input{conv_h.tikz}
        \caption{uniform refinement}
        \label{fig:funcerr-pieper-vexler:uniform}
    \end{subfigure}
    
    \begin{subfigure}[t]{\textwidth}
    \centering
        \input{conv_h_adapt.tikz}
        \caption{adaptive refinement}
        \label{fig:funcerr-pieper-vexler:adaptive}
    \end{subfigure}

    \caption{Comparison of true error and estimator} 
    \label{fig:funcerr-pieper-vexler}
\end{figure}

To illustrate \cref{prop:conv}, we consider $g:=y^\dagger$ as exact data, add Gaussian noise at different levels $\delta$, and adaptively compute the corresponding minimizer $u_{\alpha(\delta),h(\delta)}^\delta$. Specifically, we
start from a relatively large $\alpha_0=10^{-2}$ and coarse uniform mesh. In an outer loop, we then reduce the regularization parameter $\alpha_k=\alpha_0\theta^k$ for $\theta=0.6$  until the discrepancy principle \eqref{eq:discrprinc} with $\overline{\tau}=2$ is satisfied. In an inner loop, we adaptively refine the discretization according to the error estimator from \cref{prop:est_normpenalty} until the precision requirements \eqref{eq:accuracy_residual} and \eqref{eq:accuracy_cost} from \cref{prop:conv} are satisfied. The resulting residuals, regularization parameters and functional values for different noise levels are plotted in \cref{fig:plot_delta} and show a convergence rate of $\mathcal{O}(\delta)$.
\pgfplotsset{cycle list/Dark2-3}
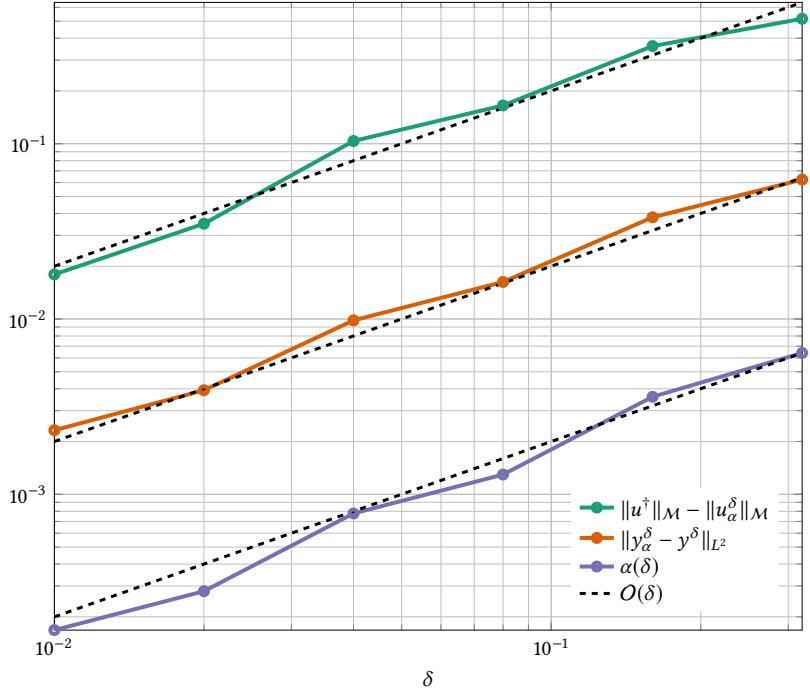
\begin{figure}
    \centering
    \input{conv_delta.tikz}
    \caption{Illustration of convergence rates as $\delta\to0$}
    \label{fig:plot_delta}
\end{figure}

\section{Conclusion}

Reliable estimators for the discretization error in Tikhonov regularization can be computed using the approach from \cite{Repin00}.
Combining this with a general result on convergence of discrete approximations and an appropriate adaptive mesh refinement strategy yields convergence of these approximations to a solution of the inverse problem.
The approach can in particular be applied to the Banach-space setting required for sparsity enhancement or Ivanov regularization.

These error estimators can be incorporated into a local refinement strategy for mesh adaptation. As shown in the examples, the estimators can be written in terms of sums over the element domains (or their interfaces) of a triangulation. Thus it makes sense to subdivide elements with relatively large contribution to the error estimator.
Note that using variational discretizations according to \cite{CCK12,Hinze05}, we do not refine independently for parameter, state, and adjoint, but use a common mesh for all three quantities.

Future research will be devoted to transferring this approach to nonlinear inverse problems via iterative linearization similarly to \cite{KaltenbacherKirchnerVeljovic13} as well as to all-at-once approaches based on the model-and-measurement formulation (\refeq{model}--\refeq{measurement}).

\appendix
\section*{Appendix}\label{app}

In this appendix we prove that the implication \eqref{abcd} holds for $\sigma$ and $\gamma$ chosen according to \eqref{rhosigma}.
Since $a+b^2\leq c+d^2$ is equivalent to $a+(b+d)^2\le 2bd+c+2d^2$ for $b,d\ge0$, the implication \eqref{abcd} is equivalent to
\begin{equation}\label{bcd}
    \forall c,d\geq0\ \forall 0\leq b\leq\sqrt{c+d^2}: \ 2bd\leq (\gamma-1)c+(\sigma-2)d^2,
    \tag{\textsc{a}.1}
\end{equation}
which (with $b=0$) can only be satisfied if $\gamma\geq1$ and $\sigma\geq2$.
It suffices to consider in \eqref{bcd} strictly positive $d$, so that upon division by $d^2$ and setting $x=\frac{b}{d}$ and $y=\frac{c}{d^2}$, the conclusion of \eqref{abcd} is equivalent to
\begin{equation}
    \sigma\geq2\ \wedge \ \gamma\geq1 \ \wedge \ \forall y\geq0\ \forall  x\in\left[0,\sqrt{1+y}\right]: \ 2x\leq (\gamma-1)y+\sigma-2.
\end{equation}
This obviously holds iff
\begin{equation}
    \sigma\geq2\ \wedge \ \gamma\geq1 \ \wedge \ \forall y\geq0: \ 2\sqrt{1+y}\leq (\gamma-1)y+\sigma-2.
\end{equation}
Setting $z=\sqrt{1+y}-1$, this is equivalent to
\begin{equation}
    \sigma\geq2\ \wedge \ \gamma\geq1 \ \wedge \ \forall z\geq0: \ 2z\leq(\gamma-1)((z+1)^2-1) +\sigma-4.
\end{equation}
For $z=0$, this implies $\sigma\geq4$, and hence \eqref{abcd} is equivalent to
\begin{equation}
    \sigma\geq4\ \wedge \ \gamma\geq1 \ \wedge \ \forall z\geq0: \ 0\leq(\gamma-1) z^2-2(2-\gamma)z+\sigma-4.
\end{equation}
We first consider the right-hand side as a quadratic polynomial in $z$, whose roots are given by $z_{\mp}=(\gamma-1)^{-1}(2-\gamma\mp\sqrt{D})$ for $D:=(2-\gamma)^2-(\gamma-1)(\sigma-4)=\gamma^2-\gamma\sigma+\sigma$. We thus arrive at the equivalent condition
\begin{equation}
    \sigma\geq4\ \wedge \ \gamma\geq1 \ \wedge \ \left( D<0 \ \vee \ \left[D\geq0\ \wedge 2-\gamma+\sqrt{D}\leq0\right]\right).
\end{equation}
Considering now  $D$ as a quadratic polynomial in $\gamma$ with roots
\begin{equation}
    \gamma_-=\frac{2\sigma}{\sigma+\sqrt{\sigma^2-4\sigma}},\qquad\quad
    \gamma_+=\frac{\sigma+\sqrt{\sigma^2-4\sigma}}{2},
\end{equation}
we arrive at
\begin{equation}
    \sigma\geq4\ \wedge \ \gamma\geq1 \ \wedge \
    \left( \gamma_-<\gamma<\gamma_+
        \ \vee \ \left[\left( \gamma\geq\gamma_+ \ \vee \ \gamma\leq\gamma_-\right)\ \wedge \
    \gamma\geq2 \ \wedge \ (\gamma-2)^2\geq D\right]\right).
\end{equation}
For $\sigma\geq4$ and $\gamma\geq1$, it can be easily checked that
$(\gamma-2)^2\geq D$ and $1<\gamma_-\leq2$ hold (the latter with strict inequality if $\sigma>4$). We have thus shown that \eqref{abcd} holds iff the conditions \eqref{rhosigma} are satisfied.

\section*{Acknowledgments}

The authors wish to thank the reviewers for their valuable comments.
Financial support by the German Science Foundation DFG and the Austrian Science Fund FWF under grants
Cl 487/1-1, Wa 3626/1-1, and I 2271 is gratefully acknowledged.
The second author is supported by the Karl Popper Kolleg ``Modeling --  Simulation -- Optimization'' funded by the Alpen-Adria-Universit\"at Klagenfurt and by the Carinthian Economic Promotion Fund (KWF).

\printbibliography
\end{document}

%% file: conv_h.tikz
\begin{tikzpicture}[scale=0.75]

\begin{axis}[%
width=\textwidth,
xmode=log,
xmin=7.30008525004821,
xmax=6595.34845809395,
xminorticks=true,
xlabel={$N$},
xminorgrids,
ymode=log,
ymin=2.11162019868825e-07,
ymax=0.136984701650354,
yminorticks=true,
yminorgrids,
axis background/.style={fill=white},
legend style={
    legend pos=south west,
    legend cell align=left,
    align=left,
    draw=none,
}
]
\addplot +[solid,mark=o,mark options={solid},line width=2pt]
  table[row sep=crcr]{%
7.30008525004821	0.0499529749284451\\
27.3015232888185	0.0139882422353981\\
105.816940999316	0.00369711987688644\\
416.860428963608	0.000937886350452271\\
1654.98105954268	0.000236726901805903\\
6595.34845809395	5.9531758774644e-05\\
};
\addlegendentry{$\langle \bar u_h,B^*(\bar w_h-\bar w)\rangle$};

\addplot +[solid,mark=o,mark options={solid},line width=2pt]
  table[row sep=crcr]{%
7.30008525004821	0.00818857382306013\\
27.3015232888185	0.00568372324552408\\
105.816940999316	0.00269948819889299\\
416.860428963608	0.000857215475515224\\
1654.98105954268	0.000231254361430355\\
6595.34845809395	5.91811995539567e-05\\
};
\addlegendentry{$(1-\kappa)\langle \bar u_h,B^*\bar w_h\rangle$};

\addplot +[solid,mark=o,mark options={solid},line width=2pt]
  table[row sep=crcr]{%
7.30008525004821	0.0316048813656292\\
27.3015232888185	0.00404455455003403\\
105.816940999316	0.000544273154162026\\
416.860428963608	9.55301826319617e-05\\
1654.98105954268	1.4206453608891e-05\\
6595.34845809395	1.77845316192209e-06\\
};
\addlegendentry{$\|\bar y_h - \hat{y}\|_{L^2}^2$};

\addplot +[solid,mark=o,mark options={solid},line width=2pt]
  table[row sep=crcr]{%
7.30008525004821	0.109080315276827\\
27.3015232888185	0.051868165600746\\
105.816940999316	0.0103087833630103\\
416.860428963608	0.00103878392622232\\
1654.98105954268	7.47717628738865e-05\\
6595.34845809395	4.89885375186705e-06\\
};
\addlegendentry{$(1-\kappa)^2 \|\bar y-g^\delta\|_{L^2}^2$};

\addplot +[solid,mark=o,mark options={solid},line width=2pt]
  table[row sep=crcr]{%
7.30008525004821	0.00479609301805109\\
27.3015232888185	0.00169466853845943\\
105.816940999316	0.000520690106164901\\
416.860428963608	0.000133729719529393\\
1654.98105954268	3.48217915142127e-05\\
6595.34845809395	8.77310180105034e-06\\
};
\addlegendentry{$|J(\bar u)-J(\bar u_h)|$};

\addplot +[solid,mark=o,mark options={solid},line width=2pt]
  table[row sep=crcr]{%
7.30008525004821	0.00108391405643421\\
27.3015232888185	0.000229964252464884\\
105.816940999316	2.70360426392873e-05\\
416.860428963608	5.4232735565149e-06\\
1654.98105954268	1.3235659832174e-06\\
6595.34845809395	2.11162019868825e-07\\
};
\addlegendentry{$\|\bar y-\bar y_h\|$};

\addplot [color=black,dashed,line width=1.5pt]
  table[row sep=crcr]{%
7.30008525004821	0.136984701650354\\
27.3015232888185	0.0366279928567047\\
105.816940999316	0.00945028263486151\\
416.860428963608	0.00239888444793425\\
1654.98105954268	0.000604236522366202\\
6595.34845809395	0.000151622011536446\\
};
\addlegendentry{$h^2$};

\end{axis}
\end{tikzpicture}%

%% file: conv_h_adapt.tikz
\begin{tikzpicture}[scale=0.75]

\begin{axis}[%
width=\textwidth,
xmode=log,
xmin=362,
xmax=16677,
xminorticks=true,
xlabel={$N$},
xminorgrids,
ymode=log,
ymin=0.82280194807475e-07,
ymax=0.0525588331227637,
yminorticks=true,
yminorgrids,
axis background/.style={fill=white},
legend style={
    legend pos=south west,
    legend cell align=left,
    align=left,
    draw=none,
}
]
\addplot +[solid,mark=o,mark options={solid},line width=2pt]
  table[row sep=crcr]{%
362	0.00596221342010821\\
627	0.00383093245762244\\
1005	0.00235510730290961\\
1384	0.00150751286298328\\
1989	0.000967225522820037\\
3542	0.000606405665000283\\
4498	0.000379968889087357\\
6971	0.000243435182003274\\
12979	0.000152938633226986\\
16677	9.59260960477648e-05\\
};
\addlegendentry{$\langle \bar u_h,B^*(\bar w_h-\bar w)\rangle$};

\addplot +[solid,mark=o,mark options={solid},line width=2pt]
  table[row sep=crcr]{%
362	0.00372984512961341\\
627	0.00276898969649714\\
1005	0.00190591210756899\\
1384	0.00130986169858669\\
1989	0.000881840914239978\\
3542	0.000571709558736518\\
4498	0.000366049628278138\\
6971	0.000237649878874842\\
12979	0.000150634901746589\\
16677	9.50146888584704e-05\\
};
\addlegendentry{$(1-\kappa)\langle \bar u_h,B^*\bar w_h\rangle$};

\addplot +[solid,mark=o,mark options={solid},line width=2pt]
  table[row sep=crcr]{%
362	0.00149355672262179\\
627	0.00130674817067999\\
1005	0.000213810439295455\\
1384	0.000186241903054152\\
1989	0.000159196748631581\\
3542	2.94055698403307e-05\\
4498	2.84345319184755e-05\\
6971	2.22439661133352e-05\\
12979	3.83395036093508e-06\\
16677	3.76436259081847e-06\\
};
\addlegendentry{$\|\bar y_h - \hat{y}\|_{L^2}^2$};

\addplot +[solid,mark=o,mark options={solid},line width=2pt]
  table[row sep=crcr]{%
362	0.0200821598276037\\
627	0.0108716220858435\\
1005	0.0051208085748293\\
1384	0.00241933919277336\\
1989	0.00109301886883349\\
3542	0.000458585772518865\\
4498	0.000187979102904968\\
6971	7.90124317684877e-05\\
12979	3.17354146117782e-05\\
16677	1.26259433009581e-05\\
};
\addlegendentry{$(1-\kappa)^2 \|\bar y-g^\delta\|_{L^2}^2$};

\addplot +[solid,mark=o,mark options={solid},line width=2pt]
  table[row sep=crcr]{%
362	0.00098463627514913\\
627	0.000509357722599607\\
1005	0.000277793680346042\\
1384	0.000270176592911467\\
1989	0.000127132544466273\\
3542	7.22948969291703e-05\\
4498	7.16272603591903e-05\\
6971	2.77298790219982e-05\\
12979	1.84112438820355e-05\\
16677	1.83132404792674e-05\\
};
\addlegendentry{$|J(\bar u)-J(\bar u_h)|$};

\addplot +[solid,mark=o,mark options={solid},line width=2pt]
  table[row sep=crcr]{%
362	8.00941254297523e-05\\
627	4.15667594411524e-05\\
1005	1.33364993335684e-05\\
1384	1.1442665890405e-05\\
1989	5.17025241723227e-06\\
3542	2.337203941507e-06\\
4498	2.29895204940308e-06\\
6971	8.78088587204e-07\\
12979	2.8433754627128e-07\\
16677	2.82280194807475e-07\\
};
\addlegendentry{$\|\bar y-\bar y_h\|$};

\addplot [color=black,dashed,line width=1.5pt]
  table[row sep=crcr]{%
362	0.0525588331227637\\
627	0.0399361531915436\\
1005	0.0315440148938256\\
1384	0.0268801665285235\\
1989	0.0224224264665438\\
3542	0.0168025703177487\\
4498	0.0149104336479388\\
6971	0.0119771215943978\\
12979	0.00877767271724973\\
16677	0.0077435665587447\\
};
\addlegendentry{${N}^{-1/2}$};

\end{axis}
\end{tikzpicture}%

%% file: conv_delta.tikz
\begin{tikzpicture}[scale=0.75]

\begin{axis}[%
width=\textwidth,
xmode=log,
xmin=0.01,
xmax=0.32,
xminorticks=true,
xmajorgrids,
xminorgrids,
xlabel={$\delta$},
ymode=log,
ymin=0.0001679616,
ymax=0.64,
yminorticks=true,
ymajorgrids,
yminorgrids,
% axis background/.style={fill=white},
legend style={
    legend pos=south east,
    legend cell align=left,
    align=left,
    draw=none,
}
]
\addplot +[solid,mark=o,line width=2pt]
  table[row sep=crcr]{%
0.01	0.0179607886998\\
0.02	0.0349696511668\\
0.04	0.1035874464367\\
0.08	0.1653315535308\\
0.16	0.3605937925276\\
0.32	0.5160658402162\\
};
\addlegendentry{$\|u^\dag\|_{\mathcal{M}}-\|u_\alpha^\delta\|_{\mathcal{M}}$};

\addplot +[solid,mark=o,line width=2pt]
  table[row sep=crcr]{%
0.01	0.002319315331044\\
0.02	0.003923789215632\\
0.04	0.009819577089955\\
0.08	0.01626412093331\\
0.16	0.03803372385517\\
0.32	0.06242364899947\\
};
\addlegendentry{$\|y_{\alpha}^\delta-y^\delta\|_{L^2}$};

\addplot +[solid,mark=o,line width=2pt]
  table[row sep=crcr]{%
0.01	0.0001679616\\
0.02	0.000279936\\
0.04	0.0007776\\
0.08	0.001296\\
0.16	0.0036\\
0.32	0.0064\\
};
\addlegendentry{$\alpha(\delta)$};

\addplot [color=black,dashed,line width=1.5pt]
  table[row sep=crcr]{%
0.01	0.02\\
0.02	0.04\\
0.04	0.08\\
0.08	0.16\\
0.16	0.32\\
0.32	0.64\\
};
\addlegendentry{$\mathcal{O}(\delta)$};

\addplot [color=black,dashed,forget plot,line width=1.5pt]
  table[row sep=crcr]{%
0.01	0.002\\
0.02	0.004\\
0.04	0.008\\
0.08	0.016\\
0.16	0.032\\
0.32	0.064\\
};
\addplot [color=black,dashed,forget plot,line width=1.5pt]
  table[row sep=crcr]{%
0.01	0.0002\\
0.02	0.0004\\
0.04	0.0008\\
0.08	0.0016\\
0.16	0.0032\\
0.32	0.0064\\
};
\end{axis}
\end{tikzpicture}%